\documentclass[11pt,a4paper]{article}
\usepackage{latexsym}
\usepackage{amsmath}
\usepackage{amssymb}
\usepackage{amscd}
\usepackage{amsthm}
\usepackage[all]{xy}

\newtheorem{thm}{Theorem}[section]
\newtheorem{cor}[thm]{Corollary}
\newtheorem{lem}[thm]{Lemma}
\newtheorem{prop}[thm]{Proposition}

\theoremstyle{definition}
\newtheorem{defn}[thm]{Definition}
\newtheorem{rem}[thm]{Remark}

\DeclareMathOperator{\GL}{\mathbf{GL}}

\DeclareMathOperator{\PGL}{\mathbf{PGL}}
\DeclareMathOperator{\Pic}{\mathbf{Pic}}
\newcommand{\Q}{\mathbb Q}
\newcommand{\R}{\mathbb R}
\newcommand{\Z}{\mathbb Z}

\newcommand{\C}{\mathbb C}

\providecommand{\vv}{\vspace{1mm}}

\newif\ifpdf \pdftrue
\ifx\pdfoutput\undefined \pdffalse \fi \ifx\pdfoutput\relax
\pdffalse \fi

\usepackage{makeidx}
\makeindex

\begin{document}

\title{Canonical key formula for projective abelian schemes}

\author{Shun Tang}

\date{}

\maketitle

\vspace{-10mm}

\hspace{5cm}\hrulefill\hspace{5.5cm} \vspace{5mm}

\textbf{Abstract.} In this paper we prove a refined version of the canonical key formula for projective abelian schemes in the sense of Moret-Bailly (cf. \cite{MB}), we also extend this discussion to the context of Arakelov geometry. Precisely, let $\pi: A\to S$ be a projective abelian
scheme over a locally noetherian scheme $S$ with unit section $e:
S\to A$ and let $L$ be a symmetric, rigidified, relatively ample
line bundle on $A$. Denote by $\omega_A$ the determinant of the
sheaf of differentials of $\pi$ and by $d$ the rank of the locally
free sheaf $\pi_*L$. In this paper, we shall prove the following
results: (i). there is an isomorphism
\begin{displaymath}
{\rm det}(\pi_*L)^{\otimes 24}\cong (e^*\omega_A^\vee)^{\otimes 12d}
\end{displaymath}
which is canonical in the sense that it is compatible with arbitrary base-change; (ii). if the generic fibre of $S$ is separated and smooth, then there exist positive integer $m$, canonical metrics on $L$ and on
$\omega_A$ such that there exists an isometry
\begin{displaymath}
{\rm det}(\pi_*\overline{L})^{\otimes 2m}\cong
(e^*\overline{\omega}_A^\vee)^{\otimes md}
\end{displaymath}
which is canonical in the sense of (i). Here the constant $m$ only depends on $g,d$ and is independent of $L$.



\section{Introduction}
Let $\pi: A\to S$ be a projective abelian scheme with unit section
$e: S\to A$, where $S$ is a normal excellent scheme. Let $L$ be a symmetric, rigidified,
relatively ample line bundle on $A$. It is well known that $L$ is
$\pi$-acyclic and $\pi_*L$ is a locally free coherent sheaf on $S$
(cf. \cite[Proposition 6.13]{MFK}). We denote by $d$ the rank of
$\pi_*L$. Moreover, denote by $\omega_A$ the determinant of the
sheaf of differentials of $\pi$. In this situation, Moret-Bailly proves that there exist a positive integer $m$ and an isomorphism
\begin{displaymath}
{\rm det}(\pi_*L)^{\otimes 2m}\cong (e^*\omega_A^\vee)^{\otimes md}
\end{displaymath}
of line bundles on $S$ (cf. \cite[Appendice 2, 1.1]{MB}). If we
write
\begin{displaymath}
\Delta(L):={\rm det}(\pi_*L)^{\otimes 2}\otimes (e^*\omega_A)^{\otimes d},
\end{displaymath}
then Moret-Bailly's result states that $\Delta(L)$ has a torsion class in the Picard group ${\rm Pic}(S)$. This is so called the key formula for projective abelian schemes, and it is denoted by ${\rm FC^{ab}}(S,g,d)$. When $S$ is a scheme which is quasi-projective over an affine noetherian scheme and $d$ is invertible on $S$, the fact that $\Delta(L)$ is a torsion line bundle is a consequence of the Grothendieck-Riemann-Roch theorem. This was
shown by Moret-Bailly and Szpiro in \cite[Appendice 2, 1.3, 1.4]{MB} and also
by Chai in his thesis \cite[V, Theorem 3.1, p. 209]{Chai1}.

Now, fixing $g,d$, we consider all such data $(A/S,L)$ in the category of locally northerian schemes, it is natural to ask if there exists canonical choice of the isomorphism
\begin{displaymath}
\alpha_L: {\rm det}(\pi_*L)^{\otimes 2m}\cong (e^*\omega_A^\vee)^{\otimes md}
\end{displaymath}
such that it is compatible with arbitrary base-change. That means if we are given a Cartesian diagram
\begin{displaymath}
\xymatrix{ A\times_S S' \ar[r]^-{p_A} \ar[d]^{\pi\times_S {\rm
id}_{S'}} & A \ar[d]^\pi \\
S' \ar[r]^f & S,}
\end{displaymath} then we always have $f^*\alpha_L=\alpha_{(p_A^*L)}$. Moret-Bailly shows in \cite[Chapitre VIII, Th\'{e}or\`{e}me 3.2]{MB} that this is true when $d$ is invertible on $S$. This is so called the canonical key formula for projective abelian schemes, and it is denoted by ${\rm FCC^{ab}}({\rm Spec}\Z[1/d],g,d)$.

\textbf{Question A:} Is there a canonical key formula for projective abelian schemes without restriction on $d$, namely ${\rm FCC^{ab}}({\rm Spec}\Z,g,d)$?

Another direction is to look for the order of $\Delta(L)$ in the Picard group ${\rm Pic}(S)$. When $S$ is a scheme which is quasi-projective over an affine noetherian scheme, Chai and
Faltings prove the following result (cf. \cite[Theorem 5.1, p. 25]{FC}).

\begin{thm}\label{101}(Chai-Faltings)
There is an isomorphism ${\rm det}(\pi_*L)^{\otimes 8d^3}\cong
(e^*\omega_A^\vee)^{\otimes 4d^4}$ of line bundles on $S$.
\end{thm}

This is to say that $4d^3$ is an upper bound of the order $\Delta(L)$ in ${\rm Pic}(S)$. Later, Chai and Faltings' result was refined by Polishchuk in \cite{Pol}. He has proved that there exists a constant $N(g)$,
which depends only on the relative dimension $g$ of $A$ over $S$,
killing $\Delta(L)$ in ${\rm Pic}(S)$. And he has also given
various bounds for $N(g)$, which depend on $d$, on $g$ and on the
residue characteristics of $S$.

In a recent work \cite{MR}, Maillot and R\"{o}ssler made a great
progress in looking for the order of $\Delta(L)$. They prove the
following results.

\begin{thm}\label{102}(Maillot-R\"{o}ssler)
(i). There is an isomorphism $\Delta(L)^{\otimes
12}\cong \mathcal{O}_S$.

(ii). For every $g\geq1$, there exist data $\pi: A\to S$ and $L$
as above such that ${\rm dim}(A/S)=g$ and such that $\Delta(L)$ is
of order $12$ in the Picard group of $S$.
\end{thm}

Now, suppose that the generic fibre of $S$
is separated and smooth, then $S$ can be viewed as an ``arithmetic scheme" in the sense of Gillet-Soul\'{e} and $A(\C)$ is a family of abelian varieties over $S(\C)$. So it is an interesting problem that studying the trivialization of some power of $\Delta(L)$ in an arithmetic sense according to the theory of Arakelov geometry. To be more precise, notice that given a K\"{a}hler fibration structure on $\pi_\C: A(\C)\to S(\C)$, any hermitian metric on $L_\C$ induces a canonical metric on the determinant bundle ${\rm det}(\pi_*L)$ i.e. the Quillen metric on the determinant (cf. Section 4.1, below). Moreover, this K\"{a}hler fibration structure implies a hermitian metric on $\Omega_\pi$, then we will get a hermitian metric on $\Delta(L)$. It will be denoted by $\Delta(\overline{L})$ the hermitian line bundle obtained in such a way.

Let us fix $g,d$ and consider all data $(A/S,L)$ such that the generic fibre of $S$ is separated and smooth, it is natural to ask the following.

\textbf{Question B:} Are there canonical metric
on $L$ and canonical K\"{a}hler fibration structure on $\pi_\C: A(\C)\to S(\C)$ such that $\Delta(\overline{L})$ has a torsion class in the arithmetic Picard group $\widehat{{\rm Pic}}(S)$?

If the answer is YES, then we get an arithmetic canonical key formula for projective abelian schemes in the context of Arakelov geometry, which can be denoted by ${\rm \widehat{FCC}^{ab}}({\rm Spec}\Z,g,d)$.

The aim of this paper is to give positive answers to \textbf{Question A} and \textbf{Question B}, actually we provide a refinement of the canonical key formula for projective abelian schemes by indicating an explicit upper bound of the order of $\Delta(L)$. Our main theorem is the following.

\begin{thm}\label{103}
Let $A/S$ be a projective abelian scheme over a locally noetherian
scheme $S$ and let $L$ be a symmetric, rigidified, relatively
ample line bundle on $A$. Then

(i). there is a trivialization of $\Delta(L)^{\otimes
12}$ which is canonical in the sense that it is compatible with
arbitrary base-change;

(ii). if the generic fibre of $S$ is separated and smooth, then there exist a positive integer $m$,
a canonical metric on $L$ and a canonical K\"{a}hler fibration structure on
$\pi_\C: A(\C)\to S(\C)$ such that there exists an isometry $\Delta(\overline{L})^{\otimes
m}\cong\overline{\mathcal{O}}_S$ which is canonical in the sense
of (i). Here the constant $m$ only depends on $g,d$ and is independent of $L$.
\end{thm}

As a byproduct of (i) of this main theorem, the condition of
quasi-projectivity on $S$ in Theorem~\ref{102} can be removed. We
also indicate that our main theorem can be viewed as a
generalization of Moret-Bailly's work \cite{MB1} where he
considered the case $d=1$.

The strategy we use to prove the first part of
Theorem~\ref{103} is the representability of a moduli functor
classifying projective abelian schemes with some additional
structures, see Section 2 below for details. And the key input in
the proof of (ii) of Theorem~\ref{103} is an arithmetic
Adams-Riemann-Roch theorem in the context of Arakelov geometry,
see \cite{Roe} or below.

\textbf{Acknowledgements.} The author wishes to thank Damian
R\"{o}ssler for providing such an interesting topic, also for his
constant encouragement and for many fruitful discussions between
them. The author is greatful to Jilong Tong for his suggestions on
the proof of a crucial lemma in this paper.

\section{Moduli functors classifying projective abelian schemes}
Until the end of this paper, all schemes will be locally
noetherian. Let $S$ be a scheme, a group scheme $\pi: A\to S$ is
called an abelian scheme if $\pi$ is smooth and proper, and the
geometric fibres of $\pi$ are connected. A basic fact (by the
rigidity lemma) is that every abelian scheme is commutative.

\subsection{Mumford's moduli functor $\mathcal{H}_{g,d,n}$}
\begin{defn}\label{201}
Let $A/S$ be an abelian scheme with unit section $e: S\to A$.\vv

(i). A line bundle $L$ on $A$ is said to be rigidified if $e^*L$
is isomorphic to $\mathcal{O}_S$;\vv

(ii). The relative Picard functor $\Pic(A/S)$ is the functor which
sends an $S$-scheme $T$ to the set of isomorphism classes of
rigidified line bundles on $A\times_S T$ with respect to the unit
section $e_T:=e\times_S {\rm id}_T$;\vv

(iii). The subfunctor $\Pic^0(A/S)$ of $\Pic(A/S)$ is the functor
which sends an $S$-scheme $T$ to the set of isomorphism classes of
rigidified line bundles $L$ on $A\times_S T$ with respect to $e_T$
such that for $\forall t\in T$, $L\otimes k(t)$ is algebraically
equivalent to zero on $A_t$.\vv
\end{defn}

We remark that if $A$ is projective over $S$, then $\Pic(A/S)$ is
represented by a separated $S$-scheme which is locally of finite
type over $S$ and $\Pic^0(A/S)$ is represented by a projective
abelian scheme over $S$. In this case, we shall call $\Pic^0(A/S)$
the dual abelian scheme of $A/S$ and denote it by $A^\vee/S$.\vv

Now let $L$ be a rigidified line bundle on a
projective abelian scheme $A/S$ with unit section $e$. Let $m:
A\times_S A\to A$ be the group law, $p_1,p_2: A\times_S A\to A$ be
the first and the second projection respectively. Consider the
line bundle
\begin{displaymath}
\widetilde{L}:=m^*(L)\otimes p_1^*(L)^{-1}\otimes p_2^*(L)^{-1}
\end{displaymath}
on $A\times_S A$. Regarding $A\times_S A$ as an abelian scheme
over $A$ with the unit section $e\times_S {\rm id}_A$, then
$\widetilde{L}$ is rigidified and for any $a\in A$,
$\widetilde{L}$ is algebraically equivalent to zero on the fibre
$A_a$ which is an abelian variety. So $\widetilde{L}$ induces an
$S$-morphism from $A$ to $A^\vee$, it is actually a group
homomorphism. We denote this homomorphism by $\lambda(L)$.

\begin{defn}\label{202}
Let $\pi: A\to S$ be a projective abelian scheme. A polarization
of $A$ is an $S$-homomorphism
\begin{displaymath}
\lambda: A\to A^\vee
\end{displaymath}
such that, for any geometric point ${\bar s}$ of $S$, the induced
morphism $\lambda_{\bar s}$ is of the form $\lambda(L_{\bar s})$
where $L_{\bar s}$ is an ample line bundle on $A_{\bar s}$.
\end{defn}

If $\lambda: A\to A^\vee$ is a polarization, $\lambda$ is finite
and faithfully flat. The pull-back of the Poincar\'{e} line bundle
along $({\rm id}_A,\lambda): A\to A\times_S A^\vee$ is a
symmetric, rigidified and relatively ample line bundle
$L^\Delta(\lambda)$ such that
$\lambda(L^\Delta(\lambda))=2\lambda$ (cf. \cite[Chapter I,
1.6]{FC} and \cite[Chapter 6, $\S$2]{MFK}).

\begin{lem}\label{203}
Let $L$ be a rigidified, relatively ample line bundle on a
projective abelian scheme $A/S$. Then the line bundle
$L^\Delta(\lambda(L))$ is canonically isomorphic to $L\otimes
[-1]^*L$.
\end{lem}
\begin{proof}
By the construction of $\lambda(L)$, the pull-back of the
Poincar\'{e} line bundle along ${\rm id}_A\times_S\lambda:
A\times_S A\to A\times_S A^\vee$ is isomorphic to $m^*(L)\otimes
p_1^*(L)^{-1}\otimes p_2^*(L)^{-1}$. Denote by $\Delta: A\to
A\times_S A$ the diagonal morphism and by $[n]: A\to A$ the
homomorphism of multiplication by $n$, then we have
\begin{align*}
L^\Delta(\lambda(L))&\cong \Delta^*\big(m^*(L)\otimes
p_1^*(L)^{-1}\otimes p_2^*(L)^{-1}\big)\\
&=[2]^*L\otimes L^{-2}\\
&\cong L\otimes [-1]^*L.
\end{align*}
The last isomorphism follows from the theorem of the cube.
\end{proof}

\begin{defn}\label{204}
Let $\pi: A\to S$ be an abelian scheme with unit section $e$. Let
$n$ be a positive integer. Assume that $A/S$ has relative
dimension $g$ and that the characteristics of the residue fields
of all $s\in S$ do not divide $n$. Then if $n\geq 2$, a
level-$n$-structure on $A/S$ is a set of $2g$ sections
$\sigma_1,\sigma_2,\ldots,\sigma_{2g}$ of $\pi$ such that\vv

(i). for all geometric points ${\bar s}$ of $S$, the images
$\sigma_i({\bar s})$ form a basis for the group of points of order
$n$ on the fibre $A_{\bar s}$;\vv

(ii). $[n]\circ \sigma_i=e$ for $i=1,2,\ldots,2g$.\vv
\\
It is convenient to call $A/S$ by itself a level-$1$-structure.
\end{defn}

\begin{defn}\label{205}
Let $g,d,n$ be three positive integers. The moduli functor
$\mathcal{A}_{g,d,n}$ is the contravariant functor from the
category of schemes to the category of sets which sends any scheme
$S$ to the set of isomorphism classes of the following data:\vv

(i). a projective abelian scheme $A$ over $S$ of relative
dimension $g$;\vv

(ii). a polarization $\lambda: A\to A^\vee$ of degree $d^2$, i.e.
$\lambda_*(\mathcal{O}_A)$ is locally free of rank $d^2$;\vv

(iii). a level-$n$-structure of $A$ over $S$.\vv
\\
We say that $(A/S,\lambda)$ is isomorphic to $(A'/S,\lambda')$ if
there exists an $S$-isomorphism of abelian schemes $\gamma: A\to
A'$ which induces an $S$-isomorphism of abelian schemes
$\gamma^\vee: A'^\vee\to A^\vee$, such that
$\lambda=\gamma^\vee\circ \lambda'\circ \gamma$. If
$\mathcal{A}_{g,d,n}$ is represented by a scheme $A_{g,d,n}$, then
$A_{g,d,n}$ will be called a fine moduli scheme.
\end{defn}

Let $\pi: A\to S$ be a projective abelian scheme of relative
dimension $g$, and let $\lambda: A\to A^\vee$ be a polarization of
$A/S$ of degree $d^2$. Then $\pi_*(L^\Delta(\lambda)^3)$ is a
locally free sheaf on $S$ of rank $6^g\cdot d$ (cf. \cite[Prop.
6.13]{MFK}). In this case, a linear rigidification of $A/S$
associated to $\lambda$ is an $S$-isomorphism
$\mathbb{P}(\pi_*(L^\Delta(\lambda)^3))\cong
\mathbb{P}_S^{6^g\cdot d-1}$.

\begin{defn}\label{206}
Let $g,d,n$ be three positive integers. The moduli functor
$\mathcal{H}_{g,d,n}$ is the contravariant functor from the
category of schemes to the category of sets which sends any scheme
$S$ to the set of isomorphism classes of the following data:\vv

(i). a projective abelian scheme $A$ over $S$ of relative
dimension $g$;\vv

(ii). a polarization $\lambda: A\to A^\vee$ of degree $d^2$;\vv

(iii). a level-$n$-structure of $A$ over $S$;\vv

(iv). a linear rigidification
$\mathbb{P}(\pi_*(L^\Delta(\lambda)^3))\cong
\mathbb{P}_S^{6^g\cdot d-1}$.\vv
\end{defn}

\begin{lem}\label{207}
Suppose that we have an $S$-isomorphism of polarized projective
abelian schemes $\gamma: (A/S,\lambda)\cong (A'/S,\lambda')$, then
$\gamma^*L^\Delta(\lambda')$ is canonically isomorphic to
$L^\Delta(\lambda)$.
\end{lem}
\begin{proof}
We only need to show that $\gamma^*L^\Delta(\lambda')\cong
L^\Delta(\gamma^\vee\circ \lambda'\circ \gamma)$ because
$\lambda=\gamma^\vee\circ \lambda'\circ \gamma$. Let
$\mathcal{P}'$ be the Poincar\'{e} line bundle on $A'\times_S
A'^\vee$, then by definition we have
\begin{align*}
\gamma^*L^\Delta(\lambda')=&\gamma^*\Delta_{A'}^*({\rm
id}_{A'}\times \lambda')^*(\mathcal{P}')\\
=&\Delta_A^*(\gamma\times\gamma)^*({\rm id}_{A'}\times
\lambda')^*(\mathcal{P}')\\
=&\Delta_A^*({\rm id}_{A'}\times
\lambda'\circ\gamma\times\gamma)^*(\mathcal{P}')\\
=&\Delta_A^*\big(\gamma\times
(\lambda'\circ\gamma)\big)^*(\mathcal{P}')
\end{align*}
On the other hand, we recall the definition of $\gamma^\vee$, it
is the scheme morphism corresponding to the natural transformation
$\alpha: {\rm Pic}^0(A'/S)\to {\rm Pic}^0(A/S)$. For any
$S$-scheme $T$, $\alpha(T)$ sends the rigidified line bundles on
$A'\times_S T$ to $A\times_S T$ by doing pull-back along
$\gamma\times {\rm id}_T$. Hence $\gamma^\vee$ corresponds to the
rigidified line bundle $(\gamma\times {\rm
id}_{A'^\vee})^*(\mathcal{P}')$ on $A\times_S A'^\vee$. Let
$\mathcal{P}$ be the Poincar\'{e} line bundle on $A\times_S
A^\vee$, then by its universal property we have $(\gamma\times
{\rm id}_{A'^\vee})^*(\mathcal{P}')\cong ({\rm id}_A\times
\gamma^\vee)^*(\mathcal{P})$. Therefore
\begin{align*}
L^\Delta(\gamma^\vee\circ \lambda'\circ
\gamma)=&\Delta_A^*\big({\rm id}_A\times (\gamma^\vee\circ
\lambda'\circ \gamma)\big)^*(\mathcal{P})\\
=&\Delta_A^*\big({\rm id}_A\times (\lambda'\circ
\gamma)\big)^*({\rm id}_A\times \gamma^\vee)^*(\mathcal{P})\\
\cong&\Delta_A^*\big({\rm id}_A\times (\lambda'\circ
\gamma)\big)^*(\gamma\times {\rm id}_{A'^\vee})^*(\mathcal{P}')\\
=&\Delta_A^*\big(\gamma\times
(\lambda'\circ\gamma)\big)^*(\mathcal{P}')
\end{align*}
So we are done.
\end{proof}

\begin{thm}\label{208}(Mumford)
For any positive integers $g,d,n$, the moduli functor
$\mathcal{H}_{g,d,n}$ is represented by a quasi-projective scheme
$H_{g,d,n}$ over $\Z$.
\end{thm}
\begin{proof}
This is \cite[Proposition 7.3]{MFK}.
\end{proof}

\subsection{The $\PGL_N$-structure on the universal abelian scheme of $\mathcal{H}_{g,d,n}$}
Define $N$ to be the integer $6^g\cdot d$, then the group scheme
$\PGL_N$ has an action on the moduli functor $\mathcal{H}_{g,d,n}$
by transforming the linear rigidification. Hence by Yoneda lemma,
$H_{g,d,n}$ admits a $\PGL_N$-action. Similarly, the universal
abelian scheme $Z_{g,d,n}$ over $H_{g,d,n}$ also admits a
$\PGL_N$-action because it represents the functor of linearly
rigidified polarized projective abelian schemes with
level-$n$-structure, and with one extra section. Although this is
a well-known fact, we don't know a reference for its proof, so we
include one.

\begin{prop}\label{209}
Let $\mathcal{Z}_{g,d,n}$ be the moduli functor from the category
of schemes to the category of sets which sends any scheme $S$ to
the set of isomorphism classes of the following data:\vv

(i). a projective abelian scheme $A$ over $S$ of relative
dimension $g$;\vv

(ii). a polarization $\lambda: A\to A^\vee$ of degree $d^2$;\vv

(iii). a level-$n$-structure of $A$ over $S$;\vv

(iv). a linear rigidification
$\mathbb{P}(\pi_*(L^\Delta(\lambda)^3))\cong
\mathbb{P}_S^{6^g\cdot d-1}$;\vv

(v). a section $\epsilon: S\to A$.\vv
\\
Then $\mathcal{Z}_{g,d,n}$ is represented by the universal abelian
scheme $Z_{g,d,n}$ of $\mathcal{H}_{g,d,n}$.
\end{prop}
\begin{proof}
We first construct a natural transformation $h$ from the functor
${\rm Hom}(\cdot,Z_{g,d,n})$ to the functor $\mathcal{Z}_{g,d,n}$,
and then we prove that $h$ is an isomorphism.

Consider the universal abelian scheme $\pi: Z_{g,d,n}\to
H_{g,d,n}$ of the moduli functor $\mathcal{H}_{g,d,n}$, then the
morphism $\pi$ corresponds to the isomorphism class of the
linearly rigidified polarized projective abelian scheme $p_2:
Z_{g,d,n}\times_{H_{g,d,n}} Z_{g,d,n}\to Z_{g,d,n}$ with
level-$n$-structure. Denote by $\Delta$ the diagonal section
$Z_{g,d,n}\to Z_{g,d,n}\times_{H_{g,d,n}} Z_{g,d,n}$.

For any scheme $U$ and any morphism $f\in {\rm Hom}(U,Z_{g,d,n})$,
we get a morphism $\pi\circ f\in {\rm Hom}(U,H_{g,d,n})$. We
define $h(f)$ to be the isomorphism class of the linearly
rigidified polarized projective abelian scheme $p_2:
Z_{g,d,n}\times_{H_{g,d,n}} U\to U$ which corresponds to the
morphism $\pi\circ f$, with the section $\Delta\times {\rm id}_U$.
This is reasonable because by construction $p_2:
Z_{g,d,n}\times_{H_{g,d,n}} U\to U$ is obtained from
$Z_{g,d,n}\times_{H_{g,d,n}} Z_{g,d,n}\to Z_{g,d,n}$ by
base-change along $f$. It is readily checked that $h$ is actually
a natural transformation from ${\rm Hom}(\cdot,Z_{g,d,n})$ to
$\mathcal{Z}_{g,d,n}$.

For the injectivity of $h$, let $f_1,f_2$ be two morphisms from
$U$ to $Z_{g,d,n}$ such that $h(f_1)=h(f_2)$. Then by the
definition of $h$, there exists an $U$-isomorphism from
$Z_{g,d,n}\times_{H_{g,d,n}} Z_{g,d,n}\times_{f_1} U$ to
$Z_{g,d,n}\times_{H_{g,d,n}} Z_{g,d,n}\times_{f_2} U$ compatible
with all of their structures. This $U$-isomorphism induces an
$U$-isomorphism $\delta$ from $Z_{g,d,n}\times_{\pi\circ f_1} U$
to $Z_{g,d,n}\times_{\pi\circ f_2} U$ compatible with all of their
structures. Hence $\pi\circ f_1=\pi\circ f_2$ and $\delta$ is
actually the identity map because such $U$-automorphism of
$Z_{g,d,n}\times_{H_{g,d,n}} U$ is already uniquely determined if
we forget about the structure of one extra section (cf. the
argument given before \cite[Prop. 7.5]{MFK}). This implies that
the two sections $\epsilon_1$ and $\epsilon_2$ from $U$ to
$Z_{g,d,n}\times_{H_{g,d,n}} U$ induced by $\Delta\times_{f_1}{\rm
id}_U$ and $\Delta\times_{f_2}{\rm id}_U$ respectively must be
equal. Now, consider the following Cartesian diagram
\begin{displaymath}
\xymatrix{Z_{g,d,n}\times_{H_{g,d,n}} Z_{g,d,n}\times_{f_1} U
\ar[r]^-{p_3} \ar[d]_{p_{12}} & U \ar[d]^{f_1} \\
Z_{g,d,n}\times_{H_{g,d,n}} Z_{g,d,n} \ar[r]^-{p_2} & Z_{g,d,n},}
\end{displaymath}
we have $p_2\circ p_{12}=f_1\circ p_3$ which implies that
$p_{12}=\Delta\circ f_1\circ p_3$ and hence $p_{12}\circ
\Delta\times_{f_1}{\rm id}_U=\Delta\circ f_1$. Since $\Delta$ is
also a section of $p_1: Z_{g,d,n}\times_{H_{g,d,n}} Z_{g,d,n}\to
Z_{g,d,n}$, so we get $p_1\circ p_{12}\circ \Delta\times_{f_1}{\rm
id}_U=f_1$. Notice that $p_1\circ p_{12}\circ
\Delta\times_{f_1}{\rm id}_U$ is indeed $p_1\circ \epsilon_1$,
hence $f_1=p_1\circ \epsilon_1$. Similarly we have $f_2=p_1\circ
\epsilon_2$, this finally implies that $f_1=f_2$ because we
already know that $\epsilon_1$ is equal to $\epsilon_2$.

For the surjectivity of $h$, let $X/U$ be a linearly rigidified
polarized projective abelian scheme with level-$n$-structure and
with one extra section $\epsilon: U\to X$. Forgetting about the
section $\epsilon: U\to X$, we get a morphism $l: U\to H_{g,d,n}$
since $H_{g,d,n}$ is a fine moduli scheme. Hence we may identify
$X/U$ with $p_2: Z_{g,d,n}\times_{H_{g,d,n}}U\to U$ with all of
their structures, the section $\epsilon: U\to X$ induces a section
$U\to Z_{g,d,n}\times_{H_{g,d,n}}U$ which is still denoted by
$\epsilon$. Define $f=p_1\circ \epsilon: U\to Z_{g,d,n}$, then it
is clear that $\pi\circ f=l$. We want to show that $h(f)$ is
exactly the isomorphism class of $p_2:
Z_{g,d,n}\times_{H_{g,d,n}}U\to U$ with the extra section
$\epsilon$. Denote by $\epsilon'$ the section of $p_2:
Z_{g,d,n}\times_{H_{g,d,n}}U\to U$ induced by the section
$\Delta\times_f{\rm id}_U$, we only need to show that
$\epsilon'=\epsilon$ since $\pi\circ f=l$ which implies the
compatibilities of other structures. In fact, in the proof of the
injectivity of $h$ we have already known that $f=p_1\circ
\epsilon'$. But $f=p_1\circ \epsilon$ follows from the definition,
so the equality $\pi\circ f=l$ immediately implies that
$\epsilon'$ and $\epsilon$ are both equal to the morphism $f\times
{\rm id}_U$. So we are done.
\end{proof}

Via Yoneda lemma, the morphism $\pi: Z_{g,d,n}\to H_{g,d,n}$
corresponds to a morphism of functors from ${\rm
Hom}(\cdot,Z_{g,d,n})$ to ${\rm Hom}(\cdot,H_{g,d,n})$. By the
construction of $h$, this functor morphism is exactly the one from
$\mathcal{Z}_{g,d,n}$ to $\mathcal{H}_{g,d,n}$ forgetting about
the structure of one extra section. Therefore, the morphism $\pi:
Z_{g,d,n}\to H_{g,d,n}$ is naturally $\PGL_N$-equivariant.\vv

Next, we investigate possible $\PGL_N$-structures on
quasi-coherent sheaves on $Z_{g,d,n}$ and $H_{g,d,n}$. Let $G$ be
a group scheme and let $X$ be a scheme, recall that an action of
$G$ on $X$ is a morphism $m_X: G\times X\to X$ which satisfies
certain properties of compatibility (cf. \cite{Ko}). Let $F$ be a
quasi-coherent sheaf on $X$, a $G$-action on $F$ is an isomorphism
of $\mathcal{O}_{G\times X}$-modules $\phi: m_X^*F\cong p_2^*F$
which satisfies the following cocycle condition on $G\times
G\times X$:
\begin{displaymath}
(p_{23}^*\phi)\circ\big(({\rm id}_G\times
m_X)^*\phi\big)=(m_G\times {\rm id}_X)^*\phi
\end{displaymath}
where $m_G: G\times G\to G$ is the multiplication of $G$. Since we
defined the $\PGL_N$-structures on $Z_{g,d,n}$ and on $H_{g,d,n}$
in a functorial way via Yoneda lemma, it is helpful to introduce
the following functorial description of the group structures of
quasi-coherent sheaves.

\begin{thm}\label{213}
Let $F$ be a quasi-coherent sheaf on $X$, then to give a $G$-structure on $F$
is equivalent to give a family of $\mathcal{O}_A$-module isomorphisms $\{\phi^A_{g,x}: (gx)^*F\cong x^*F\}$ for each affine scheme $A$ over $X$ and for all $A$-valued points $g\in G(A), x\in X(A)$, such that

(i). $\phi^A_{g,x}\circ \phi^A_{g',gx}=\phi^A_{g'g,x}$;\vv

(ii). for any $X$-morphism of affine schemes $f: B\to A$, $\phi^A_{f\circ g,f\circ x}$ is the pull-back of $\phi^A_{g,x}$ along the morphism $f$;\vv
\end{thm}
\begin{proof}
Suppose that we are given a $G$-structure on $F$, which is an
isomorphism $\phi: m_X^*F\cong p_2^*F$ satisfying certain property
of associativity (a cocycle condition). Let $A$ be an affine scheme over
$X$, for any $A$-valued points $g\in G(A)$ and $x\in X(A)$ we have
a morphism $u: A\to G\times X$. Then $gx\in X(A)$ is
the morphism $m_X\circ u$. We define $\phi^A_{g,x}:
(gx)^*F\cong x^*F$ to be the isomorphism which
is the pull-back of $\phi$ along the morphism $u$. It is readily
checked that this assignment satisfies the conditions (i) and (ii).

Conversely, suppose that we are given an assignment satisfying conditions (i) and (ii).
We choose an open affine covering $\{A_i\}_{i\in I}$ of $G\times X$, the natural embedding $u_i: A_i\to G\times X$
gives $A$-valued points $g_i\in G(A_i)$ and $x_i\in X(A)$. Then $\phi^{A_i}_{g_i,x_i}$ provides an isomorphism $u_i^*m_X^*F\cong u_i^*p_2^*F$. These isomorphisms can be glued to get a global isomorphism $\phi: m_X^*F\cong p_2^*F$ because of (ii), and $\phi$ satisfies the cocycle condition because of (i). So we get a
$G$-structures on $F$.
\end{proof}

Now, we can describe the $\PGL_N$-structure on the line bundle
$L^\Delta(\lambda)$ where $\lambda$ is the universal polarization
of the universal abelian scheme $Z_{g,d,n}$. Let $A$ be an affine
scheme over $Z_{g,d,n}$. Let $g\in \PGL_N(A)$ and $u\in
Z_{g,d,n}(A)$ be two $A$-valued points. We shall denote $gu$ by
$v$ for convenience. Then we need to give an isomorphism
$\phi^A_{v,u}: v^*L^\Delta(\lambda)\cong u^*L^\Delta(\lambda)$.
Consider the following Cartesian diagram:
\begin{displaymath}
\xymatrix{ Z_{g,d,n} \ar[d]^{\pi} &
Z_{g,d,n}\times_{H_{g,d,n}}Z_{g,d,n} \ar[d]^{p_2} \ar[l]_-{p_1} &
Z_{g,d,n}\times_{H_{g,d,n}}Z_{g,d,n}\times_v A \ar[d]^{p_3}
\ar[l]_{p_{12}} \\
H_{g,d,n} & Z_{g,d,n} \ar[l]_{\pi} & A \ar[l]_{v}.}
\end{displaymath}
In the proof of Proposition~\ref{209}, we showed that
$v^*L^\Delta(\lambda)$ is equal to $(p_1\circ p_{12}\circ
\Delta\times_{v}{\rm id}_A)^*L^\Delta(\lambda)$ where $\Delta$ is
the diagonal section from $Z_{g,d,n}$ to
$Z_{g,d,n}\times_{H_{g,d,n}}Z_{g,d,n}$. Hence
$v^*L^\Delta(\lambda)$ is canonically isomorphic to the pull-back
of the line bundle $L^\Delta(\lambda_v)$ on
$Z_{g,d,n}\times_{H_{g,d,n}}Z_{g,d,n}\times_v A$ along the section
$\Delta\times_{v}{\rm id}_A$. Similarly, $u^*L^\Delta(\lambda)$ is
canonically isomorphic to the pull-back of the line bundle
$L^\Delta(\lambda_u)$ on
$Z_{g,d,n}\times_{H_{g,d,n}}Z_{g,d,n}\times_u A$ along the section
$\Delta\times_{u}{\rm id}_A$. But notice that the action of
$\PGL_N(A)$ on $Z_{g,d,n}(A)$ is just the transformation of linear
rigidifications which doesn't affect the structures of projective
abelian scheme, of polarization, of level-$n$-structure and of the
extra one section. So there exists another projective abelian
scheme $X$ over $A$ with polarization $\lambda_X$,
level-$n$-structure and one extra section $\epsilon_X$ such that
there exist unique $A$-isomorphisms $\eta_v: X\to
Z_{g,d,n}\times_{H_{g,d,n}}Z_{g,d,n}\times_v A$ and $\eta_u: X\to
Z_{g,d,n}\times_{H_{g,d,n}}Z_{g,d,n}\times_u A$ which are
compatible with all of their structures (except the structure of
linear rigidifications!). By Lemma~\ref{207}, we have canonical
isomorphisms $\eta_v^*L^\Delta(\lambda_v)\cong
L^\Delta(\lambda_X)\cong \eta_u^*L^\Delta(\lambda_u)$. Pulling
back these isomorphisms along the section $\epsilon_X$, we finally
get an isomorphism $\phi^A_{v,u}: v^*L^\Delta(\lambda)\cong
u^*L^\Delta(\lambda)$. It is easily seen that this isomorphism
$\phi^A_{v,u}$ is independent of the choice of the representative
$X/A$. Let $g'\in \PGL_N(A)$ be another $A$-valued point. Denote
$g'v$ by $w$, we have to show that $\phi^A_{v,u}\circ
\phi^A_{w,v}=\phi^A_{w,u}$. But this simply follows from the
construction: $\eta_w^*L^\Delta(\lambda_w)\cong
L^\Delta(\lambda_X)\cong
\eta_v^*L^\Delta(\lambda_v)=\eta_v^*L^\Delta(\lambda_v)\cong
L^\Delta(\lambda_X)\cong \eta_u^*L^\Delta(\lambda_u)$ is equal to
$\eta_w^*L^\Delta(\lambda_w)\cong L^\Delta(\lambda_X)\cong
\eta_u^*L^\Delta(\lambda_u)$. At last, $\phi^A_{v,u}$ is clearly
functorial so that we get a $\PGL_N$-action on the geometric
functor $\theta(L^\Delta(\lambda))$ and hence a $\PGL_N$-structure
on $L^\Delta(\lambda)$. The $\PGL_N$-structure constructed like
this way is called the canonical $\PGL_N$-structure. If a
quasi-coherent sheaf on $Z_{g,d,n}$ comes from the data of
structures in the definition of the representable functor
$\mathcal{Z}_{g,d,n}$, then it is compatible with arbitrary
base-change and we call it a universal quasi-coherent sheaf. For
universal quasi-coherent sheaves on $Z_{g,d,n}$, it is always
possible to construct canonical $\PGL_N$-structures. For instance,
we may define the canonical $\PGL_N$-structure on the canonical
sheaf $\omega_{Z_{g,d,n}/{H_{g,d,n}}}$. \vv

Over $H_{g,d,n}$, the line bundle $\Delta(L^\Delta(\lambda))$
admits the canonical $\PGL_N$-structure because it is a universal
bundle. And the structure sheaf $\mathcal{O}_{H_{g,d,n}}$ admits a
natural $\PGL_N$-structure because $H_{g,d,n}$ is
$\PGL_N$-equivariant. Notice that the canonical
$\PGL_N$-structures on $L^\Delta(\lambda)$ and on
$\omega_{Z_{g,d,n}/{H_{g,d,n}}}$ induce a $\PGL_N$-structure on
$\Delta(L^\Delta(\lambda))$ since $\pi: Z_{g,d,n}\to H_{g,d,n}$ is
$\PGL_N$-equivariant. It can be checked that this is exactly the
canonical $\PGL_N$-structure on $\Delta(L^\Delta(\lambda))$. \vv

For any scheme $S$ and for every positive integer $k$, we shall
denote by $S[1/k]$ the scheme $S\times_\Z \{{\rm
Spec}(\Z)-\bigcup_{p\mid k}(p)\}$ obtained by removing the fibres
over $p$ which divides $k$. To end this subsection, we summarize
some facts about $H_{g,d,n}$.

\begin{thm}\label{214}
Let $g,d,n$ be any positive integers.\vv

(i). If $n>6^g\cdot d \cdot\sqrt{g!}$, then the fine moduli scheme
$A_{g,d,n}$ exists and $H_{g,d,n}$ is a $\PGL_N$-torsor over
$A_{g,d,n}$. In this case, $A_{g,d,n}$ is faithfully flat over
${\rm Spec}\Z[1/n]$ and is smooth over ${\rm Spec}\Z[1/{nd}]$.\vv

(ii). $H_{g,d,n}$ is faithfully flat over ${\rm Spec}\Z[1/n]$ and
is smooth over $\Q$.
\end{thm}
\begin{proof}
The statements in (i) are the contents of \cite[Prop. 7.6, Thm.
7.9]{MFK} and \cite[Thm 1.4 (a)]{Chai}. We prove (ii). Let $k$ be
a positive integer, suppose that $\sigma_1,\ldots,\sigma_{2g}$ is
a level-$nk$-structure on a projective abelian scheme $A/S$. Then
$k\sigma_1,\ldots,k\sigma_{2g}$ is a level-$n$-structure on $A/S$.
This defines a morphism $p_n^{(k)}$ from $\mathcal{H}_{g,d,nk}$ to
$\mathcal{H}_{g,d,n}$, and hence from $H_{g,d,nk}$ to $H_{g,d,n}$.
Moreover let $\Gamma_n$ be the group $\GL_{2g}(\Z/{n\Z})$, then
$\Gamma_n$ has a natural action on $\mathcal{H}_{g,d,n}$ and hence
on $H_{g,d,n}$. Indeed, for any element $T=(a_{i,j})\in \Gamma_n$
and any level-$n$-structure $\sigma_1,\ldots,\sigma_{2g}$ on
$A/S$, the set of sections
$\sum_{j=1}^{2g}a_{1,j}\sigma_j,\ldots,\sum_{j=1}^{2g}a_{2g,j}\sigma_j$
is also a level-$n$-structure. There is a canonical morphism from
$\Gamma_{nk}$ to $\Gamma_n$, we denote its kernel by
$\Gamma_n^{(k)}$. In \cite[Lemma 7.11]{MFK}, Mumford proved that
$p_n^{(k)}: H_{g,d,nk}\to H_{g,d,n}[1/k]$ is a
$\Gamma_n^{(k)}$-torsor and $p_n^{(k)}$ is actually a finite
\'{e}tale morphism.

Now, let $g,d,n$ be any positive integers. According to (i), we
may take an integer $k$ big enough so that $A_{g,d,nk}$ exists and
$H_{g,d,nk}$ is a $\PGL_N$-torsor over $A_{g,d,nk}$. So
$H_{g,d,nk}$ is faithfully flat over ${\rm Spec}\Z[1/{nk}]$.
Together with the faithfully flatness of $p_n^{(k)}$, we know that
$H_{g,d,n}[1/k]$ is faithfully flat over ${\rm Spec}\Z[1/{nk}]$.
Replacing $k$ by another integer big enough which is prime to $k$,
we finally obtain that $H_{g,d,n}$ is faithfully flat over ${\rm
Spec}\Z[1/n]$. For the smoothness, firstly note that
$H_{g,d,nk}\times_\Z\Q$ is smooth over $A_{g,d,nk}\times_\Z\Q$
since $\PGL_N$ is smooth over $\Q$, hence the generic fibre of
$H_{g,d,nk}$ is smooth by (i) i.e. the generic fibre of
$H_{g,d,nk}$ is regular. Again by the faithfully flatness of
$p_n^{(k)}$, we know that the generic fibre of $H_{g,d,n}$ is
regular. Therefore $H_{g,d,n}$ is smooth over $\Q$.
\end{proof}

\subsection{A variant of Mumford's moduli functor
$\widetilde{\mathcal{H}}_{g,d,n}$} In this subsection, we shall
introduce a variant of Mumford's moduli functor
$\widetilde{\mathcal{H}}_{g,d,n}$ which classifies linearly
rigidified projective abelian schemes with level-$n$-structure,
and with a symmetric, rigidified ample line bundle.

\begin{defn}\label{215}
Let $g,d,n$ be three positive integers. The moduli functor
$\widetilde{\mathcal{H}}_{g,d,n}$ is the contravariant functor
from the category of schemes to the category of sets which sends
any scheme $S$ to the set of isomorphism classes of the following
data:\vv

(i). a projective abelian scheme $\pi: A\to S$ of relative
dimension $g$, with unit section $e$;\vv

(ii). a symmetric, rigidified and relatively ample line bundle $L$
on $A$ such that the rank of the vector bundle $\pi_*(L)$ is
$d$;\vv

(iii). a rigidification $e^*L\cong \mathcal{O}_S$;\vv

(iv). a level-$n$-structure of $A$ over $S$;\vv

(v). a linear rigidification $\mathbb{P}(\pi_*(L^6))\cong
\mathbb{P}_S^{6^g\cdot d-1}$.
\end{defn}

\begin{lem}\label{216}
Let $(A/S,L)$ and $(A'/S,L')$ be two projective abelian schemes
over $S$ equipped with symmetric, rigidified ample line bundles.
Suppose that there exists an $S$-isomorphism $\gamma: A\to A'$
which induces an $S$-isomorphism $\gamma^\vee: A'^\vee\to A^\vee$
such that $L\cong \gamma^*L'$. Then $\lambda(L)$ is equal to
${\gamma^\vee}\circ \lambda(L')\circ \gamma$.
\end{lem}
\begin{proof}
We only need to show that $\lambda(L')\circ
\gamma={\gamma^\vee}^{-1}\circ \lambda(\gamma^*L')$ because $L$ is
isomorphic to $\gamma^*L'$. Consider the following diagram at
first:
\begin{displaymath}
\xymatrix{ A \ar[r]^-{\gamma} & A' \ar[rr]^-{\lambda(L')} &&
A'^\vee .}
\end{displaymath}
By the universal property of $A'^\vee$, the composition
$\lambda(L')\circ \gamma$ corresponds to the rigidified line
bundle $({\rm id}_{A'}\times \gamma)^*(m_{A'}^*L'\otimes
p_{1A'}^*L'^{-1}\otimes p_{2A'}^*L'^{-1})$ on $A'\times_S A$.

On the other hand, consider the following diagram:
\begin{displaymath}
\xymatrix{ A \ar[rr]^-{\lambda(\gamma^*L')} && A^\vee
\ar[r]^-{{\gamma^\vee}^{-1}} & A'^\vee .}
\end{displaymath}
Recall that for any $S$-scheme $T$, the morphism
${\gamma^\vee}^{-1}$ sends the elements of the relative Picard
functor ${\rm Pic}^0(A/S)(T)$ to ${\rm Pic}^0(A'/S)(T)$ by doing
pull-back along $(\gamma^{-1}\times {\rm id}_T)$. Then by the
definition of $\lambda(\gamma^*L')$ we know that the composition
${\gamma^\vee}^{-1}\circ \lambda(\gamma^*L')$ corresponds to the
rigidified line bundle
\begin{align*}
&(\gamma^{-1}\times {\rm id}_A)^*(m_{A}^*\gamma^*L'\otimes
p_{1A}^*\gamma^*L'^{-1}\otimes p_{2A}^*\gamma^*L'^{-1})\\
=&(\gamma^{-1}\times {\rm id}_A)^*(\gamma\times
\gamma)^*(m_{A'}^*L'\otimes p_{1A'}^*L'^{-1}\otimes
p_{2A'}^*L'^{-1})\\
=&(\gamma\times \gamma\circ \gamma^{-1}\times {\rm
id}_A)^*(m_{A'}^*L'\otimes p_{1A'}^*L'^{-1}\otimes
p_{2A'}^*L'^{-1})\\
=&({\rm id}_{A'}\times \gamma)^*(m_{A'}^*L'\otimes
p_{1A'}^*L'^{-1}\otimes p_{2A'}^*L'^{-1})
\end{align*}
on $A'\times_S A$. So we are done.
\end{proof}

Let $\pi: A\to S$ be a projective abelian scheme, and let $L$ be a
symmetric, rigidified ample line bundle on $A$. Then $L$ induces a
polarization $\lambda(L): A\to A^\vee$ such that
$L^\Delta(\lambda(L))$ is canonically isomorphic to $L^2$ (cf.
Lemma~\ref{203}). Notice that the square of the rank of $\pi_*(L)$
is equal to the degree of the polarization $\lambda(L)$, then
according to Lemma~\ref{216} and Lemma~\ref{207}, we have a
well-defined natural transformation $\alpha$ from
$\widetilde{\mathcal{H}}_{g,d,n}$ to $\mathcal{H}_{g,d,n}$. The
following property of rigidified line bundle is very important for our later arguments.

\begin{rem}\label{217}
Let $L_1, L_2$ be two rigidified line bundles on an abelian
scheme $\pi: A\to S$ which are isomorphic to each other. Then
there is only one isomorphism between $L_1$ and $L_2$ which is compatible
with the rigidifications.
\end{rem}

\begin{rem}\label{218}
Let $L$ be a rigidified line bundle on $A$ and let $\rho_1,
\rho_2$ be two rigidifications $e^*L\to \mathcal{O}_S$. By
Remark~\ref{217}, there exists a unique automorphism $l$ of $L$
which is compatible with $\rho_1$ and $\rho_2$. If $L$ is moreover
relatively ample and is equipped with some linear rigidification
$\mathbb{P}(\pi_*(L^6))\cong \mathbb{P}_S^{6^g\cdot d-1}$, then
the automorphism $l$ respects the linear rigidification. This
result follows from the construction of the automorphism $l$, the
projection formula and the fact that $\PGL_N$ is isomorphic to
$\GL_{N+1}/\mathbf{G}_m$.
\end{rem}

By Theorem~\ref{207}, $\mathcal{H}_{g,d,n}$ is represented by a
quasi-projective scheme $H_{g,d,n}$ over $\Z$. Consider the
category $\mathcal{C}$ consisting of schemes over $H_{g,d,n}$,
endowed with the fppf topology. Then
$\widetilde{\mathcal{H}}_{g,d,n}$ induces a moduli functor
$\widetilde{\mathcal{H}}'_{g,d,n}$ over $\mathcal{C}$ which sends
an object $f: U\to H_{g,d,n}$ to the set of elements of
$\widetilde{\mathcal{H}}_{g,d,n}(U)$ which have the same image
under $\alpha$ in $\mathcal{H}_{g,d,n}(U)$ corresponding to $f$.

\begin{prop}\label{219}
The moduli functor $\widetilde{\mathcal{H}}'_{g,d,n}$ is a sheaf
over the site $\mathcal{C}$.
\end{prop}
\begin{proof}
Let $U\to H_{g,d,n}$ be an object in $\mathcal{C}$ and let
$\{U_i\to H_{g,d,n}\}_{i\in I}$ be an fppf covering of $U$. We
need to show that the following diagram
\begin{displaymath}
\xymatrix{ \widetilde{\mathcal{H}}'_{g,d,n}(U\to H_{g,d,n}) \ar[r]
& \prod_i \widetilde{\mathcal{H}}'_{g,d,n}(U_i\to H_{g,d,n})
\ar@<+.7ex>[r]^-{p_1^*} \ar@<-.7ex>[r]_-{p_2^*} & \prod_{i,j}
\widetilde{\mathcal{H}}'_{g,d,n}(U_i\times_U U_j\to H_{g,d,n})}
\end{displaymath}
is an equalizer.

Firstly, let $a,b$ be two elements in
$\widetilde{\mathcal{H}}'_{g,d,n}(U\to H_{g,d,n})$ such that
$a\mid_{U_i}=b\mid_{U_i}$ for any $i\in I$. Since
$\alpha(a)=\alpha(b)$, we may assume that $a,b$ are represented by
$(A/U,L_1)$ and $(A/U,L_2)$ with the same level-$n$-structure and
the same linear rigidification. Then $a\mid_{U_i}=b\mid_{U_i}$
means that there exists an $U$-automorphism $\eta_i$ of $A\times_U
U_i$ such that $L_1\mid_{U_i}\cong\eta_i^*L_2\mid_{U_i}$.
Moreover, this $U$-automorphism $\eta_i$ respects all of the
structures appearing in the definition of $\mathcal{H}_{g,d,n}$,
we again use Mumford's argument given before \cite[Prop. 7.5]{MFK}
to conclude that such $U$-automorphism is unique. Hence $\eta_i$
is the identity map so that $L_1\mid_{U_i}\cong L_2\mid_{U_i}$ for
every $i\in I$. We take into account the rigidifications of $L_1$
and $L_2$, by Remark~\ref{217}, the isomorphism $L_1\mid_{U_i}\cong
L_2\mid_{U_i}$ is unique because it is compatible with the
rigidifications. Therefore this family of isomorphisms are
compatible on $U_i\times_U U_j$ so that $L_1\cong L_2$ because the
fibred category of quasi-coherent sheaves over the category of
schemes is a stack with respect to the fppf topology (cf.
\cite[Thm. 4.23]{Vi}). The resulting isomorphism $L_1\cong L_2$ is
compatible with the rigidifications since its restriction to every
$U_i$ is so. We finally have $a=b$.

Secondly, let $\prod_ia_i$ be an element in $\prod_i
\widetilde{\mathcal{H}}'_{g,d,n}(U_i\to H_{g,d,n})$ with $a_i\in
\widetilde{\mathcal{H}}'_{g,d,n}(U_i\to H_{g,d,n})$ such that
$p_1^*(\prod_ia_i)=p_2^*(\prod_ia_i)$. So we have
$p_1^*(\prod_i\alpha(a_i))=p_2^*(\prod_i\alpha(a_i))$ in
$\prod_{i,j} \mathcal{H}_{g,d,n}(U_i\times_U U_j)$. Then by the
fact that every representable functor is a sheaf with respect to
the fppf topology, we have an element $t\in
\mathcal{H}_{g,d,n}(U)$ such that $t\mid_{U_i}=\alpha(a_i)$.
Therefore we may assume that there exists a linearly rigidified
polarized projective abelian scheme $(A/U,\lambda)$ with a
level-$n$-structure and every $a_i$ is equal to $(A\times_U
U_i/{U_i},L_i)$ with the induced level-$n$-structure such that
$\lambda(L_i)=\lambda\mid_{U_i}$. Now, we take the rigidification
of $L_i$ into account and use the descent theory for
quasi-projective morphisms via relatively ample line bundles. Precisely, notice that $p_1^*(\prod_ia_i)=p_2^*(\prod_ia_i)$
and the isomorphism between $L_j$ and $L_i$ on
$A\times_U U_i\times_U U_j$ is required to be compatible with
the rigidifications, we may glue all $(A\times_U
U_i/{U_i},L_i)$ to get a scheme $P$ which is quasi-projective over $U$
and a relatively ample line bundle $L_P$ on $P$ such that there
exists a family of $U$-isomorphisms $\beta_i:
(P,L_P)\mid_{U_i}\cong (A\times_U U_i/{U_i},L_i)$ satisfying
certain condition of compatibilities. The structure morphism
$P\to U$ is moreover flat and proper since it is so after
base-change along a faithfully flat morphism.
Next, we may glue all isomorphisms $\beta_i$ to get a global $U$-morphism $\beta$ from
$P$ to $A$, this follows from the fact that the fibred category
associated to a stable class of morphisms (e.g. flat morphisms)
over the category of schemes is a prestack with respect to the
fppf topology (cf. \cite[Prop. 4.31]{Vi}). The morphism $\beta$ is actually an
isomorphism because it becomes an isomorphism after base-change
along a faithfully flat morphism. We transfer $L_P$ via $\beta$ to
get a relatively ample line bundle $L$ on $A$, then $L\mid_{U_i}$
is isomorphic to $L_i$ for any $i\in I$. Finally, we still have to
show that $L$ is symmetric and rigidified. But, this can be easily
seen from Remark~\ref{217} and the fact that the fibred category of
quasi-coherent sheaves over the category of schemes is a stack
with respect to the fppf topology (cf. \cite[Thm. 4.23]{Vi}),
because every $L_i$ is symmetric and rigidified by definition. So
we are done.
\end{proof}

Now, let $\mathcal{G}$ be the group functor of $\mathcal{C}$ which
sends an object $U\to H_{g,d,n}$ to the group of $2$-torsion
points of the dual of $Z_{g,d,n}\times_{H_{g,d,n}} U$. Then
$\mathcal{G}$ is clearly represented by the subscheme of
$2$-torsion points of the dual of $Z_{g,d,n}$. We denote this
scheme by $G$, it is finite flat over $H_{g,d,n}$ and is \'{e}tale
over $H_{g,d,n}[1/2]$.\vv

To end this subsection, we mention that $G$ has a natural action
on $\widetilde{\mathcal{H}}'_{g,d,n}$. Let $A/S$ be a projective
abelian scheme and let $L$ be a symmetric, rigidified ample line
bundle on $A$. Then for any rigidified $2$-torsion line bundle $E$
(which is automatically symmetric), $E\otimes L$ and $L$ induce
the same polarization. In fact, $L^\Delta(\lambda(E\otimes
L))=L^\Delta(\lambda(L))=L^2$ and hence $2\lambda(E\otimes
L)=2\lambda(L)$ which implies that $\lambda(E\otimes
L)=\lambda(L)$. So we may define the action of $G$ on
$\widetilde{\mathcal{H}}'_{g,d,n}$ by twisting the relatively
ample line bundle by a rigidified $2$-torsion line bundle. By
Remark~\ref{218} this action is well-defined, namely it is
independent of the choice of the explicit rigidification of a
rigidified $2$-torsion line bundle. This $G$-action will play a
crucial role in the study of the representability of
$\widetilde{\mathcal{H}}_{g,d,n}$.

\subsection{Representability of $\widetilde{\mathcal{H}}_{g,d,n}$}
In this subsection, we shall investigate the representability of
the functor $\widetilde{\mathcal{H}}_{g,d,n}$. It is clear that
$\widetilde{\mathcal{H}}_{g,d,n}$ is representable if and only if
$\widetilde{\mathcal{H}}'_{g,d,n}$ is representable. So we may
concentrate on the representability of
$\widetilde{\mathcal{H}}'_{g,d,n}$. Our first result is the
following.

\begin{lem}\label{220}
Let $\lambda$ be the universal polarization of the universal
abelian scheme $Z_{g,d,n}$ over $H_{g,d,n}$. Then there exists an
fppf covering $\{U_i\to H_{g,d,n}\}_{i\in I}$ of $H_{g,d,n}$ such
that for every $i\in I$ there exists a symmetric, rigidified ample
line bundle $L_i$ on $Z_{g,d,n}\times_{H_{g,d,n}}U_i$ as a square
root of $L^\Delta(\lambda)_{U_i}$.
\end{lem}
\begin{proof}
The proof given here is due to Tong (private communication between
Tong and the author). In the following, we shall use the notation
$A/S$ instead of $Z_{g,d,n}/{H_{g,d,n}}$. We first prove that
there exists an fppf covering of $S$ such that locally on the fppf
topology, $L^\Delta(\lambda)$ is isomorphic to $[2]^*M$ for some
rigidified line bundle $M$. It is sufficient to show that locally
on fppf topology, $L^\Delta(\lambda)$ admits an action of ${\rm
ker}([2])$ which is compatible with the action of ${\rm ker}([2])$
on $A$ given by translations.

In fact, denote by $\lambda'$ the polarization defined by
$L^\Delta(\lambda)$, then $\lambda'=2\lambda$ so that we have the
inclusion ${\rm ker}([2])\subset {\rm ker}(\lambda')$. This
implies that for any point $a\in {\rm ker}([2])$,
$t_a^*L^\Delta(\lambda)\simeq L^\Delta(\lambda)$ where $t_a: A\to
A$ is the translation map with respect to $a$. Consider the sheaf
$K(L^\Delta(\lambda))=\{(a,\alpha)\mid a\in {\rm ker}([2]),
\alpha: t_a^*L^\Delta(\lambda)\simeq L^\Delta(\lambda)\}$, it fits
a short exact sequence
\begin{displaymath}
0\to \mathbf{G}_m\to K(L^\Delta(\lambda))\to {\rm ker}([2])\to 0.
\end{displaymath}
Since the fppf sheaf $\mathcal{E}xt_{S}^1({\rm ker}([2]),\mathbf{G}_m)$ is
trivial (cf. \cite[Lemme 6.2.2]{Ray}), we may replace $S$ by an
fppf localization and suppose that the exact sequence given above
is split. Hence there exists a section $\theta: {\rm ker}([2])\to
K(L^\Delta(\lambda))$ of group schemes over $S$. This section
gives an action of ${\rm ker}([2])$ on $L^\Delta(\lambda)$ which
is compatible with the action of ${\rm ker}([2])$ on $A$ given by
translations.

Now, for any $s\in S$, we may assume that there exists an fppf
neighborhood $V$ of $s$ such that over $A_V$, $[2]^*M_V\simeq
L^\Delta(\lambda)_{A_V}$ for some rigidified line bundle $M_V$.
Since $[2]^*M_V$ is algebraically equivalent to $M_V^{\otimes4}$,
we have $N:=L^\Delta(\lambda)_{A_V}\otimes M_V^{-4}\in A_V^\vee$.
That means $N$ induces a section $\eta: V\to {\rm Pic}^0(A_V/V)$.
Consider the following Cartesian diagram
\begin{displaymath}
\xymatrix{ W \ar[r] \ar[d] & {\rm Pic}^0(A_V/V) \ar[d]^{[2]}\\
V \ar[r]^-{\eta} & {\rm Pic}^0(A_V/V),}
\end{displaymath}
here $[2]$ is faithfully flat. Therefore over $A_W$, $N$ has a
rigidified square root and hence $L^\Delta({\lambda})$ has a
square root which is rigidified by construction. We denote this
square root by $L_W$.

Now let $Q$ be the rigidified line bundle $L_W\otimes
[-1]^*L_W^\vee$. This line bundle $Q$ is a $2$-torsion since $L_W^{\otimes
2}=L^\Delta({\lambda})_{A_W}$ which is symmetric. Therefore $Q$
induces a section $\beta: W\to {\rm Pic}^0(A_W/W)$. Consider the
following Cartesian diagram
\begin{displaymath}
\xymatrix{ P \ar[r] \ar[d] & {\rm Pic}^0(A_W/W) \ar[d]^{[2]}\\
W \ar[r]^-{\beta} & {\rm Pic}^0(A_W/W),}
\end{displaymath} here
$[2]$ is again faithfully flat. So over $A_P$, $Q$ has a
rigidified square root $E$. Denote by $L_P$ the tensor product
$L_W\otimes E$, then $L_P$ is rigidified, symmetric and $L_P$
induces the same polarization ${\lambda}_{A_P}$ as $L_W$ because
$E$ is a torsion bundle. This actually implies that $E$ is a
$2$-torsion since $L_P^{\otimes
2}=L^\Delta({\lambda})_{A_P}=L_W^{\otimes 2}$ over $A_P$. So over
$A_P$, $Q$ is the trivial bundle and $L_W$ is symmetric. Finally,
to get the fppf covering $\{U_i\to S\}_{i\in I}$ we just take such
fppf neighborhoods $P$ around all the points of $S$. Notice that
$I$ can be chosen to be of finite number, because $S$ is
quasi-compact and flat morphisms are open if they are of finite
type.
\end{proof}

\begin{cor}\label{221}
$\widetilde{\mathcal{H}}'_{g,d,n}$ is a $G$-torsor sheaf over the
site $\mathcal{C}$.
\end{cor}
\begin{proof}
It is easily seen that the action of $G$ on
$\widetilde{\mathcal{H}}'_{g,d,n}$ is free and transitive. Then
$\widetilde{\mathcal{H}}'_{g,d,n}$ is a $G$-torsor sheaf if and
only if there exists an fppf covering $\{U_i\to H_{g,d,n}\}_{i\in
I}$ of $H_{g,d,n}$ such that for every $i\in I$ the set
$\widetilde{\mathcal{H}}'_{g,d,n}(U_i\to H_{g,d,n})$ is non-empty,
because $G$ is an abelian group. This fact follows from
Lemma~\ref{220}, so we are done.
\end{proof}

\begin{thm}\label{222}
The moduli functor $\widetilde{\mathcal{H}}'_{g,d,n}$ is
representable.
\end{thm}
\begin{proof}
$G$ is finite and hence affine over $H_{g,d,n}$, so the
representability of $\widetilde{\mathcal{H}}'_{g,d,n}$ follows
from Corollary~\ref{221} and \cite[Theorem III.4.3 (a)]{Mil}.
\end{proof}

From Theorem~\ref{222} we know that the moduli functor
$\widetilde{\mathcal{H}}_{g,d,n}$ is represented by a scheme
$\widetilde{H}_{g,d,n}$ over $H_{g,d,n}$. The following
proposition summarizes some properties of $\widetilde{H}_{g,d,n}$.

\begin{prop}\label{223}
$\widetilde{H}_{g,d,n}$ is flat and quasi-projective over $\Z$,
and it is smooth over $\Q$.
\end{prop}
\begin{proof}
$G$ is finite flat over $H_{g,d,n}$ and $G[1/2]$ is \'{e}tale over
$H_{g,d,n}[1/2]$, so $\widetilde{H}_{g,d,n}$ is finite flat over
$H_{g,d,n}$ and $\widetilde{H}_{g,d,n}[1/2]$ is \'{e}tale over
$H_{g,d,n}[1/2]$. This follows from \cite[Prop. III.4.2]{Mil}.
Then the statement can be deduced from Theorem~\ref{214}.
\end{proof}

To end this subsection, we mention that the scheme
$\widetilde{H}_{g,d,n}$ and the universal abelian scheme
$\widetilde{Z}_{g,d,n}$ over $\widetilde{H}_{g,d,n}$ admit natural
$\PGL_N$-actions such that the structure morphism $\pi:
\widetilde{Z}_{g,d,n}\to \widetilde{H}_{g,d,n}$ is
$\PGL_N$-equivariant. Moreover, the universal line bundle $L$ on
$\widetilde{Z}_{g,d,n}$ can be equipped with the canonical
$\PGL_N$-structure. Similarly, by changing the rigidification,
$\widetilde{H}_{g,d,n}$ and $\widetilde{Z}_{g,d,n}$ are
$\mathbf{G}_m$-equivariant schemes. But by Remark~\ref{218}, these
$\mathbf{G}_m$-actions are both trivial. The universal line bundle
$L$ on $\widetilde{Z}_{g,d,n}$ can also be equipped with the
canonical $\mathbf{G}_m$-structure which is not the trivial one.

\section{Construction of the canonical trivialization of $\Delta(L)^{\otimes 12}$}
Let $\pi: A\to S$ be a projective
abelian scheme of relative dimension $g$ with a symmetric,
rigidified ample line bundle $L$ such that the rank of $\pi_*L$ is
equal to $d$. Here $S$ is not necessarily quasi-projective over an
affine scheme. In this subsection, we shall construct an
isomorphism $\Delta(L)^{\otimes 12}\cong \mathcal{O}_S$ which is
canonical in the sense that it is compatible with arbitrary
base-change.\vv

Suppose that $\{U_i\}$ is an open covering of $S$ such that the
restriction of $A/S$ to every $U_i$ admits a linear
rigidification. It is clear that such open covering always exists
and moreover we may assume that all $U_i$ are affine. We choose a
linear rigidification for $A_{U_i}/{U_i}$, then there exists a
unique morphism $f: U_i\to \widetilde{H}_{g,d,1}$ such that $A_{U_i}/{U_i}$ is
isomorphic to $\widetilde{Z}_{g,d,1}\times_f U_i/{U_i}$ with the structure of
rigidified line bundle and the structure of linear rigidification. Write
$L_Z$ for the universal rigidified line bundle on $\widetilde{Z}_{g,d,1}$. By
Proposition~\ref{223} we know that $\widetilde{H}_{g,d,1}$ is quasi-projective
over $\Z$, then we may use the theorem of Maillot and R\"{o}ssler (cf. Theorem~\ref{102})
to conclude that the order of $\Delta(L_Z)$ in ${\rm Pic}(\widetilde{H}_{g,d,1})$ is a divisor of $12$.
We choose an arbitrary trivialization $\eta:
\Delta(L_Z)^{\otimes 12}\cong
\mathcal{O}_{\widetilde{H}_{g,d,1}}$, then it becomes universal. Pulling back $\eta$ to $U_i$ along $f$,
we get an isomorphism $\eta_i$ between
$\Delta(L_i)^{\otimes 12}$ and
$\mathcal{O}_{U_i}$. We want to show that $\eta_i$ is independent
of the choice of the linear rigidification for $A_{U_i}/{U_i}$.
Note that $A_{U_i}/{U_i}$ can be chosen as a representative to
define the canonical $\PGL_N$-structure on
$\Delta(L_Z)$, then the statement that $\eta_i$ is
independent of the choice of the linear rigidification is
equivalent to the statement that the isomorphism $\eta:
\Delta(L_Z)^{\otimes 12}\cong
\mathcal{O}_{\widetilde{H}_{g,d,1}}$ is $\PGL_N$-equivariant. But $\eta$ is
automatically $\PGL_N$-equivariant because of the following lemma.

\begin{lem}\label{301}
Every line bundle $L$ on $\widetilde{H}_{g,d,1}$ admits at most one
$\PGL_N$-structure.
\end{lem}
\begin{proof}
Let $m: \PGL_N\times \widetilde{H}_{g,d,1}\to \widetilde{H}_{g,d,1}$ be the
$\PGL_N$-action on $\widetilde{H}_{g,d,1}$. A $\PGL_N$-structure on $L$ is an
isomorphism $\gamma: m^*L\cong p_2^*L$ which satisfies certain
property of associativity. We first prove that two
$\PGL_N$-structures $\gamma_1$ and $\gamma_2$ of $L$ are equal if
they are equal on the generic fibre. Actually, by
Proposition~\ref{223} we know that $\widetilde{H}_{g,d,1}$ is flat over $\Z$
hence $\mathcal{O}_{H_{g,d,1}}$ is a flat $\Z$-module. Therefore
the restriction map from $\mathcal{O}_{\PGL_N\times \widetilde{H}_{g,d,1}}$ to
$\mathcal{O}_{\PGL_N\times \widetilde{H}_{g,d,1}}\otimes_\Z \Q$ is injective.
But $p_2^*L$ is a flat $\mathcal{O}_{\PGL_N\times
\widetilde{H}_{g,d,1}}$-module, so the restriction map $p_2^*L\to p_2^*L_\Q$
is injective. Hence if $\gamma_1$ and $\gamma_2$ are equal over
generic fibre, then they must be equal globally. By
Proposition~\ref{223}, the generic fibre of $\widetilde{H}_{g,d,1}$ is smooth
which implies that $({\widetilde{H}_{g,d,1}})_\Q$ is geometrically reduced.
Moreover, $({\PGL_N})_\Q$ is connected and there are no
non-trivial characters $\PGL_N\to \mathbf{G}_m$. Thus we can use
\cite[Prop. 1.4]{MFK} to conclude that $L_\Q$ admits only one
$\PGL_N$-structure. So we are done.
\end{proof}

Now we have known that the isomorphism $\eta_i$ is really
independent of the choice of the linear rigidification, hence
$\eta_i$ and $\eta_j$ are equal on $U_i\times_S U_j$ so that we
may glue all $\{\eta_i\}$ to get a global isomorphism $\alpha:
\Delta(L)^{\otimes 12}\cong \mathcal{O}_S$.
Clearly, the fact that $\eta_i$ is independent of the choice of
the linear rigidification also shows that the global isomorphism
$\alpha$ is independent of the choice of the open affine covering.
Therefore, this isomorphism $\alpha:
\Delta(L)^{\otimes 12}\cong \mathcal{O}_S$ is the
desired canonical trivialization of
$\Delta(L)^{\otimes 12}$.

\section{The class of $\Delta(\overline{L})$ in the arithmetic Picard group $\widehat{{\rm Pic}}(S)$}

\subsection{Arithmetic Adams-Riemann-Roch theorem}
In this subsection, we describe the arithmetic
Adams-Riemann-Roch theorem that we will use to investigate the
class of $\Delta(\overline{L})$ in the arithmetic Picard group $\widehat{{\rm Pic}}(S)$.
One can see \cite{Roe} for more details.

Let $X$ be a quasi-projective scheme over $\Z$ with smooth generic
fibre. Then $X(\C)$, the set of complex points of the variety
$X\times_\Z\C$ admits a structure of complex manifold. Arakelov
theory provides a powerful tool in the study of Diophantine
geometry by doing algebraic geometry of $X$ over $\Z$ and
hermitian complex geometry of $X(\C)$ simultaneously. For
instance, in the setting of Arakelov geometry, we have hermitian
vector bundles on $X$ and the arithmetic Grothendieck group
$\widehat{K_0}(X)$ associated to $X$, which are the main objects
in the expression of the arithmetic Adams-Riemann-Roch theorem.

\begin{defn}\label{401}
Denote by $F_\infty$ the antiholomorphic involution of $X(\C)$
induced by the complex conjugation. A hermitian vector bundle
$\overline{E}$ on $X$ is an algebraic vector bundle $E$ on $X$,
endowed with a hermitian metric on the associated holomorphic
vector bundle $E_\C$ on $X(\C)$ which is invariant under
$F_\infty$.
\end{defn}

Denote by $A^{p,p}(X)$ the set of real smooth forms $\omega$ of
type $(p,p)$ on $X(\C)$ which satisfy
$F_\infty^*\omega=(-1)^p\omega$, and by $Z^{p,p}(X)\subseteq
A^{p,p}(X)$ the kernel of the differential operator
$d=\partial+\overline{\partial}$. We shall write
$\widetilde{A}(X)$ for the set of form classes
\begin{displaymath}
\widetilde{A}(X):=\bigoplus_{p\geq0}\big(A^{p,p}(X)/({\rm
Im}\partial+{\rm Im}\overline{\partial})\big)
\end{displaymath} and
\begin{displaymath}
Z(X):=\bigoplus_{p\geq0}Z^{p,p}(X).
\end{displaymath}
To every hermitian vector bundle $\overline{E}$ on $X$, we may
associate a Chern character form ${\rm ch}(\overline{E}):={\rm
ch}(E_\C,h)$ which is defined by the Chern-Weil theory on
hermitian holomorphic vector bundles on complex manifolds.
Similarly, we have Todd form ${\rm Td}(\overline{E})$. Notice that
the Chern-Weil theory is not additive for short exact sequence of
hermitian vector bundles. Let $\overline{\varepsilon}: 0\to
\overline{E}'\to \overline{E}\to \overline{E}''\to 0$ be an exact
sequence of hermitian vector bundles on $X$, we can associate to
it a Bott-Chern secondary characteristic class $\widetilde{{\rm
ch}}(\overline{\varepsilon})\in \widetilde{A}(X)$ which satisfies
the differential equation
\begin{displaymath}
{\rm dd}^c\widetilde{{\rm ch}}(\overline{\varepsilon})={\rm
ch}(\overline{E}')-{\rm ch}(\overline{E})+{\rm
ch}(\overline{E}'')
\end{displaymath}
where ${\rm dd}^c$ is the differential
operator $\frac{\overline{\partial}\partial}{2\pi i}$.

\begin{defn}\label{402}
The arithmetic Grothendieck group $\widehat{K_0}(X)$ with respect
to $X$ is the abelian group generated by the elements of
$\widetilde{A}(X)$ and by the isometry classes of hermitian vector
bundles on $X$, modulo the following relations:

(i). for every exact sequence $\overline{\varepsilon}$ as above,
$\widetilde{{\rm
ch}}(\overline{\varepsilon})=\overline{E}'-\overline{E}+\overline{E}''$;

(ii). if $\alpha\in \widetilde{A}(X)$ is the sum of two elements
$\alpha'$ and $\alpha''$ in $\widetilde{A}(X)$, then the equality
$\alpha=\alpha'+\alpha''$ still holds in $\widehat{K_0}(X)$.
\end{defn}

We now recall the definitions of $\lambda$-ring and associated
Adams operations.

\begin{defn}\label{403}
A $\lambda$-ring is a unitary ring $R$ with operations $\lambda^k,
k\in \mathbb{N}$ satisfying the following axioms.

(i). $\lambda^0=1$, $\lambda^1(x)=x\quad \forall x\in R$,
$\lambda^k(1)=0\quad \forall k>1$.

(ii).
$\lambda^k(x+y)=\sum_{i=0}^k\lambda^i(x)\cdot\lambda^{k-i}(y)$.

(iii).
$\lambda^k(xy)=P_k(\lambda^1(x),\ldots,\lambda^k(x);\lambda^1(y),\ldots,\lambda^k(y))$
for some universal polynomial $P_k$ with integral coefficients.

(iv).
$\lambda^k(\lambda^l(x))=P_{k,l}(\lambda^1(x),\ldots,\lambda^{kl}(x))$
for some universal polynomial $P_{k,l}$ with integral
coefficients.

Putting $\lambda_t(x):=\sum_k\lambda^k(x)t^k$, we have
$\lambda_t(x+y)=\lambda_t(x)\cdot\lambda_t(y)$ by (ii). For the
definitions of $P_k$ and $P_{k,l}$, we refer to \cite[I. 4.2,
4.3]{SABK}
\end{defn}

Given a $\lambda$-ring $R$, the relationship between the Adams
operations $\psi^k$ and the $\lambda$-operations is the following.
Define a formal power series $\psi_t$ by the formula
\begin{displaymath}
\psi_t(x):=\frac{-t\cdot{\rm d}{\lambda_{-t}}(x)/{{\rm
d}t}}{\lambda_{-t}(x)}.
\end{displaymath}
The Adams operations are then given by
the identity (cf. \cite[V, Appendice]{GBI})
\begin{displaymath}
\psi_t(x)=\sum_{k\geq1}\psi^k(x)t^k.
\end{displaymath}

Now consider the group $\Gamma(X):=Z(X)\oplus \widetilde{A}(X)$,
we equip it with a grading $\Gamma(X)=\oplus_{p\geq0}\Gamma_p(X)$
where
\begin{displaymath}
\Gamma_p(X):=\left\{%
\begin{array}{ll}
    Z^{p,p}(X)\oplus\widetilde{A}^{p-1,p-1}(X), & \hbox{if $p\geq1$;} \\
    Z^{0,0}(X), & \hbox{if $p=0$.} \\
\end{array}%
\right.
\end{displaymath}
We define a bilinear map $*$ from $\Gamma(X)\times\Gamma(X)$ to
$\Gamma(X)$ by the formula
\begin{displaymath}
(\omega,\eta)*(\omega',\eta')=(\omega\wedge\omega',\omega\wedge\eta'+\eta\wedge\omega'+({\rm
dd}^c\eta)\wedge\eta').
\end{displaymath}
This map endows $\Gamma(X)$ with the
structure of a commutative graded $\R$-algebra (cf. \cite[Lemma
7.3.1]{GS2}). Hence there is a unique $\lambda$-ring structure on
$\Gamma(X)$ such that the $k$-th associated Adams operation is
given by the formula $\psi^k(x)=\sum_{i\geq0}k^ix_i$, where $x_i$
stands for the component of degree $i$ of the element $x\in
\Gamma(X)$ (cf. \cite[7.2, p. 361]{GBI}).

\begin{defn}\label{404}
If $\overline{E}+\eta$ and $\overline{E}'+\eta'$ are two
generators of $\widehat{K_0}(X)$, then we may define a product
$\otimes$ by the formula
\begin{displaymath}
(\overline{E}+\eta)\otimes(\overline{E}'+\eta')=\overline{E}\otimes\overline{E}'+[({\rm
ch}(\overline{E}),\eta)*({\rm ch}(\overline{E}'),\eta')]
\end{displaymath} where
$[\cdot]$ refers to the projection on the second component of
$\Gamma(X)$. If $k\geq0$, we set
\begin{displaymath}
\lambda^k(\overline{E}+\eta)=\wedge^k(\overline{E})+[\lambda^k({\rm
ch}(\overline{E}),\eta)]
\end{displaymath}
where $\lambda^k({\rm
ch}(\overline{E}),\eta)$ stands for the image of $({\rm
ch}(\overline{E}),\eta)$ under the $k$-th $\lambda$-operation of
$\Gamma(X)$.
\end{defn}

It was shown by Roessler in \cite{Roe1} that $\widehat{K_0}(X)$
with the product $\otimes$ and the operations $\lambda^k$ in
Definition~\ref{404} is actually a $\lambda$-ring.

Let $Y$ be another quasi-projective scheme over $\Z$ with smooth
generic fibre and suppose that $f: X\to Y$ is a flat projective
morphism which is smooth over $\Q$. In this situation, $f_\C: X(\C)\to Y(\C)$ is
a holomorphic proper submersion between complex manifolds. A K\"{a}hler fibration
structure on $f_\C$ is a real closed $(1,1)$-form $\omega$ on $X(\C)$ which induces
K\"{a}hler metrics on the fibres (cf. \cite[Def. 1.1, Thm. 1.2]{BK}).
If $f_\C$ is endowed with a K\"{a}hler fibration structure, we
may define a reasonable push-forward morphism $f_*: \widehat{K_0}(X)\to \widehat{K_0}(Y)$.
For instance, we can fix a conjugation invariant K\"{a}hler metric on
$X(\C)$ and choose corresponding K\"{a}hler form $\omega$ as the K\"{a}hler fibration structure.

Let $(E,h^E)$ be a hermitian vector bundle on $X$ such that
$E$ is $f$-acyclic i.e. the higher direct image $R^qf_*E$ vanishes
for $q>0$. By semi-continuity theorem (cf. \cite[Theorem III.12.8,
Cor. III.12.9]{Ha}) the sheaf of module $f_*E:=R^0f_*E$ is locally
free and the natural map
\begin{displaymath}
(R^0f_*E)_y\rightarrow H^0(X_y,E\mid_{X_y})
\end{displaymath} is an isomorphism
for every point $y\in Y$. In particular, we have natural isomorphism
\begin{displaymath}
(R^0f_*E_\C)_y\rightarrow H^0(X(\C)_y,E_\C\mid_{X(\C)_y})
\end{displaymath} for every point $y\in Y(\C)$.
On the other hand, we may endow
$H^0(X(\C)_y,E_\C\mid_{X(\C)_y})$ with a $L^2-$hermitian product given by the
formula
\begin{displaymath}
<s,t>_{L^2}:=\frac{1}{(2\pi)^{d_y}}\int_{X(\C)_y}h^E(s,t)\frac{\omega^{d_y}}{d_y!}
\end{displaymath}
where $d_y$ is the complex dimension of the fibre $X(\C)_y$. It can be
shown that these hermitian products depend on $y$ in a $C^\infty$
manner (cf. \cite[p.278]{BGV}) and hence define a hermitian metric
on $(f_*E)_\C$. This metric is called the $L^2$-metric. Let $(Tf_\C,h_f)$
be the relative holomorphic tangent bundle with the metric induced
by $\omega$. In \cite[Theorem 3.9]{BK}, Bismut and K\"{o}hler
constructed a smooth form $T(\omega,h^E)\in
A(Y)=\oplus_{p\geq0}A^{p,p}(Y)$ satisfying the differential
equation
\begin{displaymath}
{\rm dd}^cT(\omega,h^E)={\rm
ch}(f_*{E_\C},h^{L^2})-\int_{X(\C)/Y(\C)}{\rm ch}(E_\C,h^E){\rm Td}(Tf_\C,h_f).
\end{displaymath}
This smooth form is called the higher analytic torsion form associated
to $(E,h^E)$, $f_\C$ and $\omega$. Its definition is too long and technic,
we can not repeat it here. We just would like to mention that the
$0$-degree part i.e. the function part of $T(\omega,h^E)$ is the
famous Ray-Singer analytic torsion of $\overline{E_\C}$ on every
fibre $X(\C)_y$. The Quillen metric $\parallel\cdot\parallel_Q$ on
${\rm det}(f_*E)$ is defined by
\begin{displaymath}
\parallel\cdot\parallel_Q^2=e^{T_0(\omega,h^E)}\cdot\parallel\cdot\parallel_{L^2}^2.
\end{displaymath}

\begin{rem}\label{fibration}
It was shown in \cite[Corollary 8.10]{BFL} that the analytic torsion form $T(\omega,h^E)$ is compatible (up to exact $\partial$- and $\overline{\partial}$-forms) with any base-change of K\"{a}hler fibration. Therefore, the Quillen metric $\parallel\cdot\parallel_Q$ on ${\rm det}(f_*E)$ and the hermitian line bundle $\Delta(\overline{L})$ are compatible with arbitrary base-change.
\end{rem}

\begin{defn}\label{405}
The push-forward morphism $f_*: \widehat{K_0}(X)\to
\widehat{K_0}(Y)$ is defined as follows.

(i). for hermitian holomorphic vector bundle $(E,h)$ on $X$ such
that $E$ is $f$-acyclic,
$f_*(E,h):=(f_*E,h^{L^2})-T(\omega,h^E)$;

(ii). for $\eta\in \widetilde{A}(X)$, $f_*\eta:=\int_{X(\C)/Y(\C)}{\rm
Td}(Tf_\C,h_f)\eta$.
\end{defn}

\begin{rem}\label{406}
$f_*$ is a well-defined group homomorphism and it satisfies the
projection formula.
\end{rem}

The next important object appearing in the expression of the
Adams-Riemann-Roch theorem is the following $R$-genus.

\begin{defn}\label{407}
The $R$-genus is the unique additive characteristic class defined
for a line bundle $L$ by the formula
\begin{displaymath}
R(L)=\sum_{m \text{ }{\rm odd},
\geq1}\big(2\zeta'(-m)+\zeta(-m)(1+\frac{1}{2}+\cdots+\frac{1}{m})\big)\frac{{\rm
c}_1(L)^m}{m!}
\end{displaymath} where $\zeta(s)$ is the Riemann zeta-function.
\end{defn}

To state the arithmetic Adams-Riemann-Roch theorem, we still need
the Bott's cannibalistic classes. For any $\lambda$-ring $R$,
denote by $R_{\rm fin}$ its subset of elements of finite
$\lambda$-dimension. For each $k\geq1$, the Bott's cannibalistic
class $\theta^k$ is uniquely determined by the following
properties

(i). $\theta^k$ maps $R_{\rm fin}$ into $R_{\rm fin}$ and the
equation $\theta^k(a+b)=\theta^k(a)\theta^k(b)$ holds for all
$a,b\in R_{\rm fin}$;

(ii). $\theta^k$ is functorial with respect to $\lambda$-ring
morphisms;

(iii). if $e$ is a line element (its $\lambda$-dimension is $1$),
then $\theta^k(e)=\sum_{i=0}^{k-1}e^i$.

Now, consider the graded commutative group
$\widetilde{A}(X)=\oplus_{p\geq0}\widetilde{A}^{p,p}(X)$, giving
degree $p$ to differential forms of type $(p,p)$. We define
$\phi^k(\omega)=\sum_{i=0}^\infty k^i\omega_i$ where $\omega_i$ is
the component of degree $i$ of $\omega\in \widetilde{A}(X)$. Then
one can compute that $\psi^k(\omega)=k\cdot\phi^k(\omega)$ where
on the left hand side $\omega$ is regarded as an element of the
$\lambda$-ring $\Gamma(X)$.

Let $\overline{E}$ be a hermitian vector bundle on $X$, then the form $k^{-{\rm rk}(E)}{\rm Td}^{-1}(\overline{E})\phi^k({\rm Td}(\overline{E}))$ is by construction a universal polynomial in the Chern forms $c_i(\overline{E})$. The associated symmetric polynomial in $r={\rm rk}(E)$ variables is denoted by $CT^k$, and one can compute that
\begin{displaymath}
CT^k=k^r\prod_{i=1}^r\frac{e^{T_i}-1}{T_ie^{T_i}}\cdot\frac{kT_ie^{kT_i}}{e^{kT_i}-1}
\end{displaymath}
where $T_1,\ldots,T_r$ are the variables. For an exact sequence of hermitian holomorphic vector bundles $\overline{\varepsilon}: 0\to \overline{E'}\to \overline{E}\to \overline{E''}\to 0$ on a complex manifold, the Bott-Chern secondary characteristic class associated to $\overline{\varepsilon}$ and to $CT^k$ will be denoted by $\widetilde{\theta}^k(\overline{\varepsilon})$.

We now turn back to the flat projective morphism $f: X\to Y$, which is smooth over $\Q$. Suppose that $f$ is a local complete intersection morphism. Let $i: X\to P$ be a regular immersion and $p: P\to Y$ be a smooth morphism, such that $f=p\circ i$. Endow $P$ with a K\"{a}hler metric and the normal bundle $N_{P/X}$ with some hermitian metric. Denote by $\overline{\mathcal{N}}$ be the exact sequence $0\to \overline{Tf_\C}\to \overline{TP_\C}\to \overline{N}_{P(\C)/X(\C)}\to 0$.

\begin{defn}\label{bott}
The arithmetic Bott class $\theta^k(\overline{Tf}^\vee)^{-1}$ of $f$ is the element $\theta^k(\overline{N}_{P/X}^\vee)\widetilde{\theta}^k(\overline{\mathcal{N}})+\theta^k(\overline{N}_{P/X}^\vee)\theta^k(i^*\overline{TP}^\vee)^{-1}$
in $\widehat{K_0}(X)[1/k]$.
\end{defn}

\begin{rem}\label{bott-class}
The arithmetic Bott class of $f$ depends neither on $i$ nor on the metrics on $P$ and on $N_{P/X}$ (cf. \cite[Lemma 3.5]{Roe}).
\end{rem}

\begin{thm}\label{408}(arithmetic Adams-Riemann-Roch)
Let $f: X\to Y$ be as above. For each $k\geq1$, let
$\theta_A^k(\overline{Tf}^\vee)^{-1}=\theta^k(\overline{Tf}^\vee)^{-1}\cdot(1+R(Tf_\C)-k\cdot\phi^k(R(Tf_\C)))$.
Then for the map $f_*: \widehat{K_0}(X)[1/k]\to
\widehat{K_0}(Y)[1/k]$, the equality
\begin{displaymath}
\psi^k(f_*(x))=f_*(\theta_A^k(\overline{Tf}^\vee)^{-1}\cdot\psi^k(x))
\end{displaymath}
holds in $\widehat{K_0}(Y)[1/k]$ for all $k\geq1$ and $x\in
\widehat{K_0}(X)[1/k]$.
\end{thm}
\begin{proof}
This is \cite[Theorem 3.6]{Roe}.
\end{proof}

\subsection{The $\gamma$-filtration of arithmetic $K_0$-theory}
Let $X$ be a quasi-projective scheme over $\Z$ with smooth generic fibre as in last subsection. In this subsection, we shall recall the $\gamma$-filtration of the $\lambda$-ring $\widehat{K_0}(X)$ and prove some basic facts.

Recall that the $\gamma$-operations on a $\lambda$-ring are defined by the formula
\begin{displaymath}
\gamma_t(x)=\sum_{i\geq0}\gamma^i(x)t^i:=\lambda_{t/(1-t)}(x).
\end{displaymath}
By construction, the $\gamma^i$ also define a pre-$\lambda$-ring structure on $\widehat{K_0}(X)$: that is, for all positive integers $k$ we have $\gamma^0(x)=1$, $\gamma^1(x)=x$ and
\begin{displaymath}
\gamma^k(x+y)=\sum_{i=0}^k\gamma^i(x)\gamma^{k-i}(y).
\end{displaymath}
Moreover, it follows from the definition that if $u$ is the class of a hermitian line bundle on $X$, then $\gamma_t(u-1)=1+(u-1)t$ and $\gamma_t(1-u)=\sum_{i\geq0}(1-u)^it^i$. This implies that $\gamma^i(u-1)=0$ for $i>1$ and $\gamma^i(1-u)=(1-u)^i$ for $i\geq0$.

Now, for any generator $(\overline{E},\eta)$ of $\widehat{K_0}(X)$, define $\varepsilon(\overline{E},\eta)={\rm rk}(E)$. This map extends to an augmentation on $\widehat{K_0}(X)$, namely a $\lambda$-ring homomorphism from $\widehat{K_0}(X)$ to $\Z$. We then construct the $\gamma$-filtration $F^n\widehat{K_0}(X) (n\geq0)$ of $\widehat{K_0}(X)$ as follows. For $n=0$, $F^0\widehat{K_0}(X):=\widehat{K_0}(X)$, for $n=1$, $F^1\widehat{K_0}(X):={\rm ker}(\varepsilon)$ and for $n\geq2$, $F^n\widehat{K_0}(X)$ is defined to be the additive subgroup generated by the elements $\gamma^{r_1}(x_1)\cdots\gamma^{r_s}(x_s)$, where $x_1,\ldots,x_s\in F^1\widehat{K_0}(X)$ and $\sum_{i=1}^sr_i\geq n$. Therefore $F^0\widehat{K_0}(X)\supseteq F^1\widehat{K_0}(X)\supseteq F^2\widehat{K_0}(X)\supseteq\cdots$ and it is easily checked that the $F^n\widehat{K_0}(X)$ are ideals that form a ring filtration. We shall denote by ${\rm Gr}^i\widehat{K_0}(X)$ the quotient group $F^i\widehat{K_0}(X)/{F^{i+1}\widehat{K_0}(X)}$.

\begin{prop}\label{gamma2}
Let $j\geq1$ and let $n\geq0$ be an integer. If $x\in F^n\widehat{K_0}(X)$, then
\begin{displaymath}
\psi^j(x)=j^nx\quad{\rm mod}\quad F^{n+1}\widehat{K_0}(X).
\end{displaymath}
\end{prop}
\begin{proof}
This statement is actually correct for any augmented $\lambda$-ring, see the last lemma in \cite[p. 96]{RSS}.
\end{proof}

\begin{defn}\label{nilpotent}
The $\gamma$-filtration of an augmented $\lambda$-ring $R$ is called locally nilpotent, if for every $x\in F^1R$, there exists a number $N(x)\in \mathbb{N}$, depending on $x$, such that $\gamma^{r_1}(x)\cdots\gamma^{r_d}(x)=0$ whenever $\sum_{i=1}^dr_i>N(x)$. It is called nilpotent, if there exists a number $N\in \mathbb{N}$, such that $F^nR=0$ for all $n>N$.
\end{defn}

It was shown by Roessler in \cite[Prop. 4.5]{Roe} that the $\gamma$-filtration of $\widehat{K_0}(X)$ is locally nilpotent, hence $\widehat{K_0}(X)$ fulfills the conditions in the following proposition.

\begin{prop}\label{gamma3}
Let $R$ be an augmented $\lambda$-ring with locally nilpotent $\gamma$-filtration. Then for any $n\geq0$,
\begin{displaymath}
F^nR_\Q=\bigoplus_{i=n}^\infty V_i
\end{displaymath}
where $V_i$ is the $k^i$-eigenspace of $\psi^k$ on $R_\Q$, $k>1$, and $V_i$ does not depend on $k$.
\end{prop}
\begin{proof}
This is in complete analogy to the proof of \cite[p. 97, Theorem 1]{RSS}.
\end{proof}

\begin{cor}\label{nilpotent2}
If $n>{\rm dim}(X)$, then $F^n\widehat{K_0}(X)_\Q=0$ .
\end{cor}
\begin{proof}
Firstly notice that the $\gamma$-filtration of the algebraic Grothendieck group $K_0(X)$ is nilpotent, precisely $F^nK_0(X)=0$ whenever $n$ is greater than the dimension of $X$. By construction the forgetful map $\widehat{K_0}(X)\to K_0(X)$ is a $\lambda$-ring morphism, then any element $x\in F^n\widehat{K_0}(X)$ is represented by a smooth form $\eta$ if $n>{\rm dim}X$. We claim that $\eta=0$ in $F^n\widehat{K_0}(X)_\Q$, which implies the statement in this corollary. Indeed, write $\eta=\sum_{i\geq0}\eta^{(i)}$ where $\eta^{(i)}$ is the $i$-th component of $\eta$, by the definition of the $\lambda$-ring structure, we know that $\eta^{(i)}\in V_{i+1}$. Since $\eta^{(i)}=0$ when $i>{\rm dim}X(\C)$, we have $\eta\in\bigoplus_{i=1}^{{\rm dim}X} V_i$. Then our claim follows from Proposition~\ref{gamma3}
\end{proof}

Now, we introduce a truncated arithmetic Chern character
\begin{displaymath}
\widehat{{\rm ch}}: \widehat{K_0}(X)[1/k]\to
\Z[1/k]\oplus\widehat{{\rm Pic}}(X)[1/k].
\end{displaymath}
This Chern character is an abelian group
homomorphism defined as follows: (i). for a hermitian vector
bundle $\overline{E}$ on $X$, $\widehat{{\rm
ch}}(\overline{E}/{k^t})=({\rm rk}(E)/{k^t},{\rm
det}(\overline{E})^{1/{k^t}})$; (ii). for an element $\omega \in
\widetilde{A}(X)$, $\widehat{{\rm
ch}}(\omega/{k^t})=(0,(\mathcal{O}_X,\omega)^{1/{k^t}})$
where $(\mathcal{O}_X,\omega)$ stands for
the trivial bundle with the metric given by $\parallel
1\parallel^2=e^{-\omega_0}$. By using \cite[Prop. 1.2.5]{GS2}, one can immediately check that this definition is compatible with the generating relation of $\widehat{K_0}(X)$. Moreover, let us introduce the paring
\begin{displaymath}
(r_1/{k^{t_1}},m_1^{1/{k^{l_1}}})\bullet(r_2/{k^{t_2}},m_2^{1/{k^{l_2}}}):=(r_1r_2/{k^{t_1+t_2}},m_2^{r_1/{k^{t_1+l_2}}}\otimes
m_1^{r_2/{k^{t_2+l_1}}})
\end{displaymath}
in the group $\Z[1/k]\oplus\widehat{{\rm Pic}}(X)[1/k]$.
The paring $\bullet$ makes this group into a commutative ring. It
can be shown that the arithmetic Chern character is a ring
homomorphism, by the properties of the determinant and by the
definition of the arithmetic Grothendieck group.

In particular, by composing with the projection to the second factor, we get a group homomorphism
\begin{displaymath}
{\rm det}: \widehat{K_0}(X)\to \widehat{{\rm Pic}}(X).
\end{displaymath}
The main result of this subsection is the following.

\begin{thm}\label{gamma}
The morphism ${\rm det}$ induces an isomorphism ${\rm Gr}^1(\widehat{K_0}(X))\cong\widehat{{\rm Pic}}(X)$.
\end{thm}

To prove this theorem, we need the following lemmas.

\begin{lem}\label{vanish}
Let $x\in \widehat{K_0}(X)$ which is represented by a smooth form $\eta$. Then for any $k\geq2$, the function part of $\gamma^k(\eta)$ vanishes.
\end{lem}
\begin{proof}
It is well known that the $\lambda$-operations $\lambda^i$ and corresponding Adams operations $\psi^k$ are related by the following Newton formula
\begin{displaymath}
\psi^k(x)-\lambda^1(x)\psi^{k-1}(x)+\cdots+(-1)^{k-1}\lambda^{k-1}(x)\psi^1(x)=(-1)^{k+1}k\lambda^k(x).
\end{displaymath}
Then by the construction of the ring structure of $\Gamma(X)$, the function part of $\lambda^k(x)$ is $(-1)^{k+1}\eta^{(0)}$.
Next, we know that the relation between the $\gamma$-operations and $\lambda$-operations is
\begin{displaymath}
\gamma^k(x)=\sum_{j=1}^k\binom{k-1}{j-1}\lambda^j(x).
\end{displaymath}
Then our lemma follows from the combinatorial identity
\begin{displaymath}
\sum_{j=1}^k\binom{k-1}{j-1}(-1)^{j+1}=\sum_{j=0}^{k-1}\binom{k-1}{j}(-1)^{j}=(1-1)^{k-1}=0.
\end{displaymath}
\end{proof}

\begin{rem}\label{vanish2}
Actually, we conjecture that for any $k\geq2$, the $i$-th component of $\gamma^k(\eta)$ vanishes if $i<k-1$.
\end{rem}

\begin{lem}\label{vanish3}
Let $x\in \widehat{K_0}(X)$ which is represented by a smooth form $\eta$ and suppose that the function part of $\eta$ vanishes, then $x\in F^2\widehat{K_0}(X)$.
\end{lem}
\begin{proof}
We claim that there exists an element $\omega\in \widetilde{A}(X)$ such that $\gamma^2(\omega)=\eta$. Indeed, we need to solve the equation $\omega+\lambda^2(\omega)=\eta$ which is equivalent to $\omega+\frac{1}{2}\omega\wedge {\rm dd}^c\omega-\frac{1}{2}\psi^2(\omega)=\eta$ by the Newton formula. But taking $\omega^{(0)}$ to be any real valued smooth function (eg. the constant function $0$) the above equation can be certainly solved by solving $\omega^{(i)}$ $(i\geq1)$ one by one. This process will terminate after finitely many steps because the dimension of $X(\C)$ is finite.
\end{proof}

\begin{proof}(of Theorem~\ref{gamma})
We firstly prove that for any $x\in F^2\widehat{K_0}(X)$ we have ${\rm det}(x)=1$, then the morphism ${\rm det}: F^1\widehat{K_0}(X)/{F^2\widehat{K_0}(X)}\to {\rm Pic}(X)$ is well-defined. If $x$ is the product of two generators of ${\rm ker}(\varepsilon)$, i.e. if $x=\gamma^1(\overline{E}_1-\overline{F}_1+\eta_1)\gamma^1(\overline{E}_2-\overline{F}_2+\eta_2)$ where $\overline{E}_i, \overline{F}_i$ are hermitian vector bundles on $X$ such that ${\rm rk}(E_i)={\rm rk}(F_i)$ and $\eta_i\in \widetilde{A}(X)$ ($i=1,2$), then it is readily checked that $\widehat{{\rm ch}}(x)=(0,1)$ and hence ${\rm det}(x)=1$ in $\widehat{{\rm Pic}}(X)$. Notice that $\widehat{{\rm ch}}$ is a ring homomorphism, the $\gamma^k$ define a pre-$\lambda$-ring structure on $\widehat{K_0}(X)$ and the function part of $\gamma^k(\eta)$ vanishes (by Lemma~\ref{vanish}), we may reduce our proof to the case where $x=\gamma^i(\overline{E}-\overline{F})$ with ${\rm rk}(E)={\rm rk}(F)$ and $i\geq2$. By the splitting principle (cf. \cite[Theorem 4.1]{Roe1}), we may assume that ${\rm rk}(E)={\rm rk}(F)=1$, then $x=\gamma^i((\overline{E}-1)+(1-\overline{F}))$ and we may furthermore reduce the proof to the case $x=\gamma^i(1-\overline{F})$ with $i\geq2$ since $\gamma^i(\overline{E}-1)=0$ for $i>1$. In this case, the statement that ${\rm det}(x)=1$ is correct because $\gamma^i(1-\overline{F})=(1-\overline{F})^i$ for $i\geq0$.

Now, let $g$ be a map from $\widehat{{\rm Pic}}(X)$ to ${\rm Gr}^1(\widehat{K_0}(X))$ which sends $\overline{L}$ to $\overline{L}-1$ mod $F^2\widehat{K_0}(X)$. This is a group homomorphism because
\begin{displaymath}
(\overline{L}\overline{L}'-1)-(\overline{L}-1)-(\overline{L}'-1)=(\overline{L}-1)(\overline{L}'-1)
\end{displaymath}
which is an element in $F^2\widehat{K_0}(X)$. Moreover, since ${\rm det}(\overline{L}-1)={\rm det}(\overline{L})=\overline{L}$, we have ${\rm det}\circ g={\rm Id}$. This implies that ${\rm det}: {\rm Gr}^1(\widehat{K_0}(X))\to\widehat{{\rm Pic}}(X)$ is surjective.

Finally, we prove that ${\rm det}$ is injective. Let $x\in F^1\widehat{K_0}(X)$, if $x$ is represented by a smooth form $\eta$, then by \cite[Prop. 1.2.5]{GS2} we have $g\circ {\rm det}(\eta)=\eta^{(0)}$ mod $F^2\widehat{K_0}(X)$. But $\eta^{(0)}=\eta$ mod $F^2\widehat{K_0}(X)$ by Lemma~\ref{vanish3}, so $g\circ {\rm det}(\eta)={\rm Id}$ for elements in $\widetilde{A}(X)$. Next, we assume that $x=\overline{E}-\overline{F}$ such that ${\rm rk}(E)={\rm rk}(F)$. By the splitting principle, we can write $x=\sum n_i\overline{L_i}=\sum n_i(\overline{L_i}-1)$ where $n_i\in \Z, \sum n_i=0$ and $\overline{L_i}$ are hermitian line bundles on $X$. Then ${\rm det}(x)=\prod \overline{L_i}^{n_i}$ and $g\circ {\rm det}(x)=x$ mod $F^2\widehat{K_0}(X)$.
\end{proof}

\subsection{The class of $\Delta(\overline{L})$ in $\widehat{{\rm Pic}}(\widetilde{H}_{g,d,1})$}
In this subsection, we shall
complete the proof of our main theorem, Theorem~\ref{103}.
Let $\pi: \widetilde{Z}_{g,d,1}\to \widetilde{H}_{g,d,1}$ be
the universal abelian scheme of the moduli functor
$\widetilde{\mathcal{H}}_{g,d,1}$ with the universal rigidified
relatively ample line bundle $L$. As an arithmetic extension of Section 3., it is clear that we only need to
show that there exist some $\PGL_N$-invariant hermitian metric on $L$ and some $\PGL_N$-invariant K\"{a}hler fibration structure on $\pi_\C$ such that
$\Delta(\overline{L})$ has a torsion class in the arithmetic Picard group $\widehat{{\rm Pic}}(\widetilde{H}_{g,d,1})$ and the bound of its order can be chosen to be independent of $L$.

Let $e: \widetilde{H}_{g,d,1}\to \widetilde{Z}_{g,d,1}$ be the
unit section and let $\eta: e^*L\cong
\mathcal{O}_{\widetilde{H}_{g,d,1}}$ be the universal
rigidification of $L$. Denote by $\gamma: [-1]^*L\cong L$ the
isomorphism which is compatible with the rigidification. In
\cite[Prop. 2.1, p. 48]{MB}, Moret-Bailly shows that any line bundle on an abelian
variety $A$ over $\C$ admits a unique hermitian metric
such that its curvature form is translation invariant and such that it is compatible with the rigidification.
Moret-Bailly's proof relies on the cubical structure on $L$, which is an isomorphism
\begin{displaymath}
\alpha: \mathcal{D}_3(L):=\bigotimes_{\emptyset\neq I\subset\{1,2,3\}}(m_I^*L)^{\otimes(-1)^{{\rm Card}(I)}}\longrightarrow \mathcal{O}_{A\times A\times A}
\end{displaymath}
satisfying some symmetry and cocycle conditions (cf. \cite[Definition 2.4.5, p. 19]{MB}),
here $m_I: A\times A\times A$ is the morphism $(x_1,x_2,x_3)\mapsto \Sigma_{i\in I}x_i$.
Endowing $\mathcal{O}_{A\times A\times A}$ with the trivial metric and endowing $\mathcal{D}_3(L)$
with the hermitian metric such that $\alpha$ is an isometry. Moret-Bailly actually proves that
the metric on $\mathcal{D}_3(L)$ obtained in this way arises from a unique hermitian metric on $L$.
This argument is also valid for the relative setting. That means, to $(\widetilde{Z}_{g,d,1}/{\widetilde{H}_{g,d,1}}, L)$, there exists a unique metric on $L_\C$ such that the first Chern form ${\rm c}_1(\overline{L})$ is
translation invariant on the fibres and such that $\eta$ is an isometry. We
endow $L_\C$ with this metric. By unicity, this metric is $\PGL_N$-invariant, because the $\PGL_N$-action only changes the linear rigidification, it doesn't affect the structures of projective abelian scheme and of the rigidified relatively ample line bundle. Also by the unicity of this metric, $\gamma: [-1]^*\overline{L}\cong\overline{L}$ is an isometry.
Moreover, by construction, this metric is compatible with the theorem of the cube, so
we have isometry $[k]^*\overline{L}\cong\overline{L}^{k^2}$ for
any $k\geq1$. Next, notice that the real $(1,1)$-form ${\rm c}_1(\overline{L})$
is positive on every fibre because $L$ is a relatively ample line bundle. Then
${\rm c}_1(\overline{L})$ defines a hermitian metric on the relative tangent bundle
$T\pi$, and hence a K\"{a}hler fibration structure on $\pi_\C$ (cf. \cite[Theorem 1.2]{BK}). This K\"{a}hler fibration structure is $\PGL_N$-invariant because the metric on $L$ is so. We endow $\Omega_\pi$ and $\omega_\pi$ with the metrics induced by the metric on $T\pi$. Finally, since ${\rm c}_1(\overline{L})$ is translation invariant on the fibres, we have a canonical isometry $\pi^*e^*\overline{\Omega}_\pi\cong\overline{\Omega}_\pi$.

We shall use the arithmetic
Adams-Riemann-Roch theorem to prove that there exist a positive integer $m$ and an isometry
$\Delta(\overline{L})^m\cong
\overline{\mathcal{O}}_{\widetilde{H}_{g,d,1}}$. We compute in
$\widehat{K_0}(\widetilde{H}_{g,d,1})[1/k]$:
\begin{align*}
\psi^{k^2}(\pi_*\overline{L})&=\pi_*(\theta_A^{k^2}(\overline{\Omega}_\pi)^{-1}\cdot
\psi^{k^2}(\overline{L})) \\
&=\pi_*(\theta_A^{k^2}(\overline{\Omega}_\pi)^{-1}\overline{L}^{k^2}) \\
&=\pi_*(\theta_A^{k^2}(\overline{\Omega}_\pi)^{-1}[k]^*\overline{L}) \\
&=\pi_*([k]^*\overline{L})\theta_A^{k^2}(e^*\overline{\Omega}_\pi)^{-1}.
\end{align*}

In other words, we have the identity
\begin{displaymath}
\theta_A^{k^2}(e^*\overline{\Omega}_\pi)\psi^{k^2}(\pi_*\overline{L})=\pi_*([k]^*\overline{L})
\end{displaymath}
which holds in $\widehat{K_0}(\widetilde{H}_{g,d,1})[1/k]$. We now apply the
arithmetic Chern character $\widehat{{\rm ch}}$ to the above identity.

To simplify the expression, we shall replace the multiplicative
notation $\otimes$ by the additive notation $+$ in the group
$\widehat{{\rm Pic}}$. Moreover, by the splitting principle we may suppose that
$e^*\overline{\Omega}_\pi=\overline{Q}_1+\cdots+\overline{Q}_g$ in
$\widehat{K_0}(\widetilde{H}_{g,d,1})$, where
$\overline{Q}_1,\ldots,\overline{Q}_g$ are hermitian line bundles.
So we have
\begin{align*}
\widehat{{\rm
ch}}(\theta_A^{k^2}(e^*\overline{\Omega}_\pi))&=\widehat{{\rm
ch}}(\theta^{k^2}(e^*\overline{\Omega}_\pi))\\
&=(k^2+\frac{k^2(k^2-1)}{2}{\rm
det}(\overline{Q}_1))\bullet\cdots\bullet(k^2+\frac{k^2(k^2-1)}{2}{\rm
det}(\overline{Q}_g))\\
&=k^{2g}+\frac{k^2(k^2-1)k^{2g-2}}{2}{\rm
det}(e^*\overline{\Omega}_\pi).
\end{align*}
and
\begin{align*}
\widehat{{\rm
ch}}(\theta_A^{k^2}(e^*\overline{\Omega}_\pi))\bullet\widehat{{\rm
ch}}(\psi^{k^2}(\pi_*\overline{L}))&=(k^{2g}+\frac{k^2(k^2-1)k^{2g-2}}{2}{\rm
det}(e^*\overline{\Omega}_\pi))\bullet(d+k^2{\rm
det}_Q(\pi_*\overline{L}))\\
&=k^{2g}d+k^{2g+2}{\rm
det}_Q(\pi_*\overline{L})+\frac{dk^2(k^2-1)k^{2g-2}}{2}{\rm
det}(e^*\overline{\Omega}_\pi).
\end{align*}

On the other hand, we have
\begin{displaymath}
\widehat{{\rm ch}}(\pi_*([k]^*\overline{L})=dk^{2g}+{\rm
det}_Q(\pi_*([k]^*\overline{L})).
\end{displaymath}
This follows from the fact
that the degree of the isogeny $[k]$ on $\widetilde{Z}_{g,d,1}$ is
$k^{2g}$ and the fact that the rank of $\pi_*([k]^*L)$ is
$dk^{2g}$ (cf. \cite[Theorem 2, p. 121]{Mum}). Finally,
multiplying by $k^{-2g}$ and specializing to $\widehat{{\rm
Pic}}(\widetilde{H}_{g,d,1})[1/k]$, we get an identity
\begin{displaymath}
k^2{\rm det}_Q(\pi_*\overline{L})+\frac{d(k^2-1)}{2}{\rm
det}(e^*\overline{\Omega}_\pi)=k^{-2g}{\rm
det}_Q(\pi_*([k]^*\overline{L}))
\end{displaymath}
in $\widehat{{\rm
Pic}}(\widetilde{H}_{g,d,1})[1/k]$.

We now compare ${\rm det}_Q(\pi_*\overline{L})$ with ${\rm
det}_Q(\pi_*([k]^*\overline{L}))$, we need the following lemmas.

\begin{lem}\label{1}
Let $\overline{\varepsilon}$ be the exact sequence $0\to 0\to \overline{T\pi}\to [k]^*\overline{T\pi}\to 0$ of hermitian vector bundles on $\widetilde{Z}_{g,d,1}(\C)$, and let $\eta$ be the smooth form $\widetilde{{\rm Td}}(\overline{\varepsilon}){\rm Td}^{-1}(\overline{T\pi})$. Then for any positive integer $l$, the identity $\psi^l([k]_*\overline{O}_{\widetilde{Z}_{g,d,1}}-[k]_*\eta)=[k]_*\overline{O}_{\widetilde{Z}_{g,d,1}}-[k]_*\eta$ holds in
$\widehat{K_0}(\widetilde{Z}_{g,d,1})[1/l]$.
\end{lem}
\begin{proof}
We apply the arithmetic Adams-Riemann-Roch theorem to the isogeny $[k]$ which is \'{e}tale over $\C$. The point is to compute $\theta_A^l(\overline{T[k]}^\vee)^{-1}$ which is $\theta^l(\overline{T[k]}^\vee)^{-1}$ since the relative tangent bundle of $[k]_\C$ vanishes. This computation can be done by using a $\theta^l(\cdot)^{-1}$-version of \cite[Prop. 1. (ii), p. 504]{GS5}, an essentially same reasoning shows that
\begin{displaymath}
\theta^l(\overline{T[k]}^\vee)^{-1}=\theta^l(\overline{T\pi}^\vee)^{-1}\cdot[k]^*\theta^l(\overline{T\pi}^\vee)-[k]^*\theta^l(\overline{T\pi}^\vee)\cdot \widetilde{\theta}^l(\overline{\varepsilon}).
\end{displaymath}
Following from the fact that $\pi^*e^*\overline{\Omega}_\pi\cong\overline{\Omega}_\pi$ and that $\pi\circ [k]=\pi$ we obtain
\begin{displaymath}
\theta^l(\overline{T[k]}^\vee)^{-1}=1-\theta^l(\overline{T\pi}^\vee)\cdot \widetilde{\theta}^l(\overline{\varepsilon}).
\end{displaymath}
Next, by \cite[Lem. 6.11, Prop. 7.3]{Roe}, we have
\begin{displaymath}
{\rm ch}(\theta^l(\overline{T\pi}^\vee))=l^g{\rm Td}(\overline{T\pi})\phi^l({\rm Td}^{-1}(\overline{T\pi}))
\end{displaymath}
and
\begin{align*}
\widetilde{\theta}^l(\overline{\varepsilon})&=[l^{-g}{\rm Td}^{-1}(\overline{T\pi})\phi^l({\rm Td}(\overline{T\pi}))l^g\widetilde{{\rm Td}}(\overline{\varepsilon})-l\phi^l(\widetilde{{\rm Td}}(\overline{\varepsilon}))]l^{-g}{\rm Td}^{-1}(\overline{T\pi})\\
&=l^{-g}{\rm Td}^{-2}(\overline{T\pi})\phi^l({\rm Td}(\overline{T\pi}))\widetilde{{\rm Td}}(\overline{\varepsilon})-l^{1-g}{\rm Td}^{-1}(\overline{T\pi})\phi^l(\widetilde{{\rm Td}}(\overline{\varepsilon}))
\end{align*}
So we finally have $\theta^l(\overline{T\pi}^\vee)\cdot \widetilde{\theta}^l(\overline{\varepsilon})={\rm ch}(\theta^l(\overline{T\pi}^\vee))\widetilde{\theta}^l(\overline{\varepsilon})$ which is nothing but $\eta-\psi^l(\eta)$.
This implies that $\psi^l([k]_*\overline{O}_{\widetilde{Z}_{g,d,1}})=[k]_*\overline{O}_{\widetilde{Z}_{g,d,1}}-[k]_*(\eta-\psi^l(\eta))$ which holds in $\widehat{K_0}(\widetilde{Z}_{g,d,1})[1/l]$.
Notice that ${\rm dd}^c\widetilde{{\rm Td}}(\overline{\varepsilon})=0$ and hence ${\rm dd}^c\eta=0$, then $\psi^l([k]_*\eta)=[k]_*(\psi^l(\eta))$. This also can be seen from the fact that the $i$-th component of $[k]_*\eta$ is $[k]_*(\eta^{(i)})$. Removing $[k]_*(\psi^l(\eta))$ to the left-hand side, we are done.
\end{proof}

\begin{lem}\label{2}
Let $\eta$ be the smooth form defined in Lemma~\ref{1}. The element $[k]_*\overline{O}_{\widetilde{Z}_{g,d,1}}$ is equal to $k^{2g}+[k]_*\eta$ in $\widehat{K_0}(\widetilde{Z}_{g,d,1})_\Q$.
\end{lem}
\begin{proof}
Write $x:=[k]_*\overline{O}_{\widetilde{Z}_{g,d,1}}-k^{2g}-[k]_*\eta$, then $x\in F^1\widehat{K_0}(\widetilde{Z}_{g,d,1})$. Take any integer $l>1$, we have $\psi^l(x)-lx\in F^2$ by Proposition~\ref{gamma2}. Then $(1-l)x\in F^2\widehat{K_0}(\widetilde{Z}_{g,d,1})[1/l]$ by Lemma~\ref{1}. Repeating such approach by using Proposition~\ref{gamma2} and Lemma~\ref{1}, for any positive integer $n>0$, we get a polynomial \begin{displaymath}
P_n(X):=\prod_{i=1}^n(1-X^i)
\end{displaymath}
such that $P_n(l)\neq 0$ and $P(l)x\in F^{n+1}\widehat{K_0}(\widetilde{Z}_{g,d,1})[1/l]$. Taking $n$ to be sufficiently large, we deduce our lemma from Corollary~\ref{nilpotent2}.
\end{proof}

\begin{lem}\label{3}
For any element $x\in \widehat{K_0}(\widetilde{Z}_{g,d,1})$, we have $\pi_*(x)-\pi_*[k]_*(x)=-\pi_*(x\cdot\eta)$ in
$\widehat{K_0}(\widetilde{H}_{g,d,1})$.
\end{lem}
\begin{proof}
If $x$ is represented by a smooth form, then the statement is clearly true. So we may suppose that $x$ is represented by a $\pi$-acyclic hermitian vector bundle $\overline{E}$. In this case, notice that $[k]$ is a finite morphism, so $\overline{E}$ is $[k]$-acyclic. Then the desired identity for $x=\overline{E}$ follows from \cite[(0.5), p. 543]{Ma}.
\end{proof}

\begin{prop}\label{compare}
The identity ${\rm det}_Q(\pi_*([k]^*\overline{L}))=k^{2g}{\rm
det}_Q(\pi_*\overline{L})$ holds in $\widehat{{\rm
Pic}}(\widetilde{H}_{g,d,1})_\Q$.
\end{prop}
\begin{proof}
Using Lemma~\ref{2}, Lemma~\ref{3} and the fact that ${\rm dd}^c\eta=0$, we compute
\begin{align*}
\pi_*([k]^*\overline{L})&=\pi_*[k]_*([k]^*\overline{L}\otimes
\overline{O}_{\widetilde{Z}_{g,d,1}})-\pi_*([k]^*\overline{L}\cdot\eta)\\
&=\pi_*(\overline{L}\otimes
[k]_*\overline{O}_{\widetilde{Z}_{g,d,1}})-\pi_*[k]_*([k]^*\overline{L}\cdot\eta)\\
&=\pi_*(k^{2g}\cdot\overline{L})+\pi_*(\overline{L}\cdot[k]_*\eta)-\pi_*(\overline{L}\cdot[k]_*\eta)\\
&=k^{2g}\cdot\pi_*\overline{L}
\end{align*}
which holds in $\widehat{K_0}(\widetilde{H}_{g,d,1})_\Q$. Applying the morphism $\widehat{{\rm det}}$ to both two sides, we get the desired identity.
\end{proof}

Thanks to Proposition~\ref{compare}, we finally conclude that
\begin{displaymath}
(k^2-1)\cdot{\rm det}_Q(\pi_*\overline{L})+\frac{d(k^2-1)}{2}\cdot{\rm
det}(e^*\overline{\Omega}_\pi))=0
\end{displaymath} in $\widehat{{\rm
Pic}}(\widetilde{H}_{g,d,1})_\Q$. In other words,
$\Delta(\overline{L})=0$ in $\widehat{{\rm Pic}}(\widetilde{H}_{g,d,1})_\Q$, which means that
$\Delta(\overline{L})$ has a torsion class in $\widehat{{\rm Pic}}(\widetilde{H}_{g,d,1})$.

We claim that the bound of the order of $\Delta(\overline{L})$ in $\widehat{{\rm Pic}}(\widetilde{H}_{g,d,1})$ can be chosen to be independent of $L$. Indeed, Lemma~\ref{2} has nothing to do with $L$, then we may choose positive integer $n_k$ (only depends on $k$) such that
\begin{displaymath}
n_k\cdot(k^2-1)\cdot{\rm det}_Q(\pi_*\overline{L})+\frac{d(k^2-1)}{2}\cdot{\rm
det}(e^*\overline{\Omega}_\pi))=0
\end{displaymath} in $\widehat{{\rm
Pic}}(\widetilde{H}_{g,d,1})[1/k]$. This means
$\Delta(\overline{L})^{n_k(k^2-1)}$ is a $k^\infty$-torsion in
$\widehat{{\rm Pic}}(\widetilde{H}_{g,d,1})$. Let $k=2$, we see
that $\Delta(\overline{L})^{3n_2}$ is a $2^\infty$-torsion in
$\widehat{{\rm Pic}}(\widetilde{H}_{g,d,1})$. Let $k=3$, we see
that $\Delta(\overline{L})^{8n_3}$ is a $3^\infty$-torsion in
$\widehat{{\rm Pic}}(\widetilde{H}_{g,d,1})$. Hence
$\Delta(\overline{L})^{24n_2n_3}$ is actually trivial in
$\widehat{{\rm Pic}}(\widetilde{H}_{g,d,1})$.

\begin{rem}\label{order}
According to the same reasoning as above, an explicit bound of the order of $\Delta(\overline{L})$ in $\widehat{{\rm Pic}}(\widetilde{H}_{g,d,1})$ can be determined if one can show that $F^n\widehat{K_0}(\widetilde{Z}_{g,d,1})$ vanishes for an effective sufficiently large number $n$.
\end{rem}

\hspace{5cm} \hrulefill\hspace{5.5cm}

Max-Planck Institute for Mathematics, Office 222, Vivatsgasse 7,
53111 Bonn, Germany

E-mail: ningxiaotang@gmail.com

\end{document}